\documentclass[10pt]{article}
\pdfoutput=1
\usepackage[utf8]{inputenc}
\usepackage{amsmath,amstext,amsthm,amsfonts}
\usepackage{graphicx}
\usepackage[sort,numbers]{natbib}
\usepackage[twoside,
  paperwidth=210mm,
  paperheight=297mm,
  textheight=622pt,
  textwidth=468pt,
  centering,
  headheight=50pt,
  headsep=12pt,
  footskip=18pt,
  footnotesep=24pt plus 2pt minus 12pt,
  columnsep=2pc]{geometry}

\DeclareMathOperator{\Min}{Min}
\newcommand{\MMin}[1]{\Min(#1)}
\newcommand{\mmin}[1]{\min(#1)}
\DeclareMathOperator{\Max}{Max}
\newcommand{\MMax}[1]{\Max(#1)}
\newcommand{\mmax}[1]{\max(#1)}

\newcommand{\bbN}{\mathbb{N}}

\newcommand{\frR}{\mathfrak{R}}
\newcommand{\frS}{\mathfrak{S}}

\newcommand{\tpt}{$\mathbf{2\!+\!2}$}
\newcommand{\tptp}{(\tpt)}
\newcommand{\tpo}{$\mathbf{3\!+\!1}$}
\newcommand{\tpop}{(\tpo)}
\newcommand{\tpto}{(\tpt,\,\tpo)}
\newcommand{\tptn}{(\tpt,\,\N)}
\newcommand{\N}{$\mathbf{N}$}

\newcommand{\Co}{\mathcal{C}^{I}} 
\newcommand{\Ct}{\mathcal{C}^{II}} 
\newcommand{\D}{\mathcal{D}} 
\newcommand{\F}{\mathcal{F}}
\newcommand{\M}{\mathcal{M}} 

\newcommand{\co}{{C}^{I}}
\newcommand{\ct}{{C}^{II}} 

\newcommand{\cl}{\prec}
\newcommand{\cP}{\mathcal{P}}
\newcommand{\cA}{\mathcal{A}} 

\DeclareMathOperator{\asc}{asc}
\DeclareMathOperator{\des}{des}
\DeclareMathOperator{\ret}{ret}
\DeclareMathOperator{\pea}{pea}

\newcommand{\pat}{\raisebox{-0.2em}{\includegraphics[height=1em]{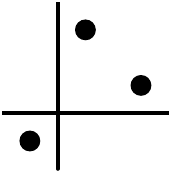}}}
\newcommand{\ppat}{\left(\pat\right)}
\DeclareMathOperator{\LRmax}{LRmax}
\DeclareMathOperator{\RLmax}{RLmax}
\DeclareMathOperator{\LRmin}{LRmin}
\DeclareMathOperator{\RLmin}{RLmin}
\DeclareMathOperator{\Av}{Av}

\newtheorem{theorem}{Theorem}[section]
\newtheorem{lemma}[theorem]{Lemma}
\newtheorem{observation}[theorem]{Observation}
\newtheorem{fact}[theorem]{Fact}
\newtheorem{proposition}[theorem]{Proposition}
\newtheorem{conjecture}[theorem]{Conjecture}

\theoremstyle{definition}
\newtheorem{definition}[theorem]{Definition}

\theoremstyle{remark}
\newtheorem{problem}{Problem}

\title{Catalan Pairs and Fishburn Triples}

\author{Vít Jelínek\thanks{
Computer Science Institute, Faculty of Mathematics and Physics,
Charles University, Malostranské náměstí 25, 118 00, Prague, 
\texttt{jelinek@iuuk.mff.cuni.cz}.
Supported by project CE-ITI (GBP202/12/G061) of the
Czech Science
Foundation.}
}

\begin{document}
\maketitle

\begin{abstract}
Disanto, Ferrari, Pinzani and Rinaldi have introduced the concept of
\emph{Catalan pair}, which is a pair of partial orders $(S,R)$ satisfying
certain axioms. They have shown that Catalan pairs provide a natural description
of objects belonging to several classes enumerated by Catalan numbers. 

In this paper, we first introduce another axiomatic structure $(T,R)$, which we
call the \emph{Catalan pair of type 2}, which describes certain Catalan
objects that do not seem to have an easy interpretation in terms of the original
Catalan pairs.

We then introduce \emph{Fishburn triples}, which are relational structures
obtained as a direct common generalization of the two types of Catalan pairs.
Fishburn triples encode, in a natural way, the structure of objects enumerated
by the Fishburn numbers, such as interval orders or Fishburn matrices. This
connection between Catalan objects and Fishburn objects allows us to associate
known statistics on Catalan objects with analogous statistics of Fishburn
objects. As our main result, we then show that several known equidistribution
results on Catalan statistics can be generalized to analogous results for
Fishburn statistics.
\end{abstract}


\section{Introduction}
The Catalan numbers $C_n=\frac{1}{n+1}\binom{2n}{n}$ are one of the most
ubiquitous number sequences in enumerative combinatorics. As of May 2013,
Stanley's Catalan Addendum~\cite{addendum} includes over 200 examples of
classes of combinatorial objects enumerated by Catalan numbers, and more
examples are constantly being discovered. 

The Fishburn numbers $F_n$ (sequence A022493 in OEIS~\cite{oeis}) are another
example of a counting sequence that arises in several seemingly unrelated
contexts. The first widely studied combinatorial class enumerated by Fishburn
numbers is the class of interval orders, also known as \tptp-free posets. Their
study was pioneered by
Fishburn~\cite{FishburnPrvni,FishburnDM,fishburn1985interval}. Later, more
objects counted by Fishburn numbers were identified, including
non-neighbor-nesting matchings~\cite{Stoimenow}, ascent sequences~\cite{BMCDK},
Fishburn matrices~\cite{dukes2010ascent,DJK}, several classes of
pattern-avoiding permutations~\cite{BMCDK,bivincular} or pattern-avoiding
matchings~\cite{Levande}.

It has been observed by several authors that various Fishburn classes contain
subclasses enumerated by Catalan numbers~\cite{DisantoAAM,CMS,DS,KimRoush}.
For instance, the Fishburn class of \tptp-free posets contains the subclasses
of \tpto-free posets and \tptn-free posets, which are both enumerated
by Catalan numbers. The aim of this paper is to describe a close relationship
between Catalan and Fishburn classes, of which the above-mentioned inclusions
are a direct consequence. 

At first sight, Catalan numbers and Fishburn numbers do not seem to have much
in common. The Catalan numbers can be expressed a simple formula
$C_n=\frac{1}{n+1}\binom{2n}{n}$, they have a simply exponential
asymptotic growth $C_n=\Theta(4^n n^{-3/2})$, and admit an algebraic generating
function, namely $(1-\sqrt{1-4x})/2$. In contrast, no simple formula for the
Fishburn numbers $F_n$ is known, their growth is superexponential
($F_n=\Theta(n!(6/\pi)^n\sqrt{n})$ as shown by Zagier~\cite{Zagier}), and their
generating function $\sum_{n\ge0}\prod_{k=1}^n (1-(1-x)^k)$, derived by
Zagier~\cite{Zagier}, is not even D-finite~\cite{BMCDK}.

Nevertheless, we will show that there is a close combinatorial relationship
between families of objects counted by Catalan numbers and those counted by
Fishburn numbers. To describe this relationship, it is convenient to represent
Catalan and Fishburn objects by relational structures satisfying certain
axioms. 

An example of such a structure are the so-called Catalan pairs, introduced by
Disanto, Ferrari, Pinzani and Rinaldi~\cite{DisantoAAM,DisantoENDM} (see
also~\cite{Bilotta}). In this paper, we first introduce another
Catalan-enumerated family of relational structures which we call Catalan pairs
of type 2. Next, we introduce a common generalization of the two types of
Catalan pairs, which we call Fishburn triples. 

This interpretation of Catalan objects and Fishburn objects by means of
relational structures allows us to detect a correspondence between known
combinatorial statistics on Catalan objects and their Fishburn counterparts.
This allows us to discover new equidistribution results for Fishburn statistics,
inspired by analogous previously known results for statistics of Catalan
objects.

In this paper, we need a lot of preparation before we can state and prove the
main results. After recalling some basic notions related to posets and 
relational structures (Subsection~\ref{ssec-relations}), we introduce 
interval orders and Fishburn matrices, and characterize \N-free 
and \tpop-free interval orders in terms of their Fishburn matrices 
(Subsection~\ref{ssec-matrices}). In Section~\ref{sec-catpairs}, we define 
Catalan pairs of type 1 and 2, which are closely related to \N-free and 
\tpop-free interval orders, respectively. We then observe that several familiar
Catalan statistics have a natural interpretation in terms of these pairs. 
Finally, in Section~\ref{sec-fish}, we introduce Fishburn triples, which are a 
direct generalization of Catalan pairs. We then state and prove our main 
results, which, informally speaking, show that certain equidistribution results 
on Catalan statistics can be generalized to Fishburn statistics.  

\subsection{Orders and Relations}\label{ssec-relations} 

A \emph{relation} (or, more properly, a \emph{binary relation}) on a set $X$ is
an arbitrary subset of the Cartesian product $X\times X$. For a relation
$R\subseteq X\times X$ and a pair of elements $x,y\in X$, we write $xRy$
as a shorthand for $(x,y)\in R$. The \emph{inverse} of $R$, denoted by $R^{-1}$,
is the relation that satisfies $xR^{-1}y$ if and only if $yRx$. Two elements
$x,y\in X$ are \emph{comparable by $R$} (or \emph{$R$-comparable}) if at least
one of the pairs $(x,y)$ and $(y,x)$ belongs to $R$.

A relation $R$ on $X$ is said to be \emph{irreflexive} if no element of $X$ is
comparable to itself. 
A relation $R$ is \emph{transitive} if $xRy$ and $yRz$ implies $xRz$ for each
$x,y,z\in X$. An irreflexive transitive relation is a \emph{partial order}. A
set $X$ together with a partial order relation $R$ on $X$ form a \emph{poset}.

An element $x\in X$ is \emph{minimal} in a relation $R$ (or $R$-minimal for
short) if there is no $y\in X$ such that~$yRx$. Maximal elements are defined
analogously. This definition agrees with the standard notion of minimal elements
in partial orders, but note that we will use this notion even when $R$ is not a
partial order. We let $\MMin{R}$ denote the set of minimal elements of $R$, and
$\mmin{R}$ its cardinality. Similarly, $\MMax{R}$ is the set of $R$-maximal
elements and $\mmax{R}$ its cardinality.

A \emph{relational structure} on a set $X$ is an ordered tuple
$\frR=(R_1,R_2,\dotsc,R_k)$ where each $R_i$ is a relation on $X$. The
relations $R_i$ are referred to as the \emph{components} of $\frR$, and the
cardinality of $X$ is the \emph{order} of~$\frR$. Given $\frR$ as above, and
given another relational structure $\frS=(S_1,\dotsc,S_m)$ on a set $Y$, we say
that \emph{$\frR$ contains $\frS$} (or $\frS$ is a \emph{substructure} of
$\frR$) if $k=m$ and there is an injection $\phi\colon Y\to X$ such that for
every $i=1,\dotsc,k$ and every $x,y\in Y$, we have $xS_i y$ if and only if
$\phi(x)R_i\phi(y)$. In such case we call the mapping $\phi$ an \emph{embedding}
of $\frS$ into $\frR$. If $\frR$ does not contain $\frS$, we say that $\frR$
\emph{avoids} $\frS$, or that $\frR$ is \emph{$\frS$-free}.
We say that $\frR$ and $\frS$ are \emph{isomorphic} if there is an embedding of
$\frS$ into $\frR$ which maps $Y$ bijectively onto~$X$. 

By a slight abuse of terminology, we often identify a relation $R$ of $X$ with a
single-component relational structure $\frR=(R)$. Therefore, e.g., saying
that a relation $R$ \emph{avoids} a relation $S$ means that the relational
structure $(R)$ avoids the relational structure $(S)$. 

\begin{figure}
\hfil\includegraphics[width=0.6\textwidth]{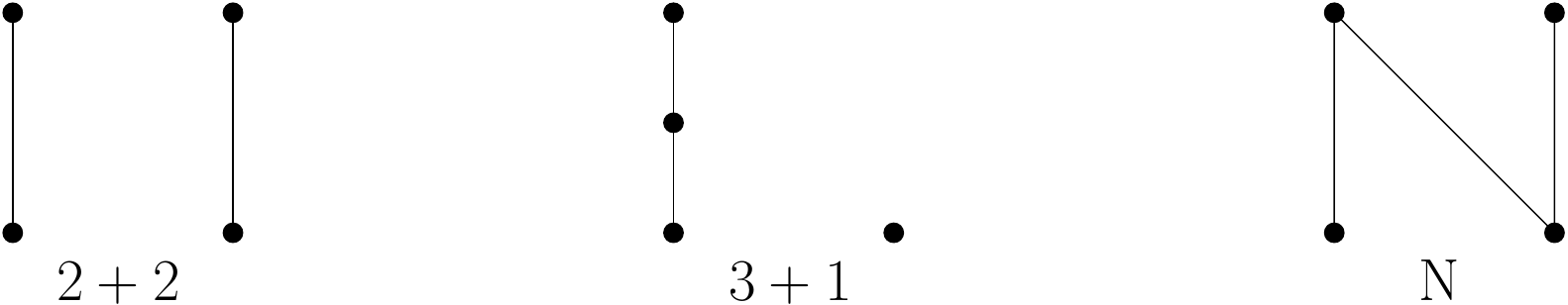}
\caption{The posets \tpt\ (left), \tpo\ (center), and \N\ (right), which will
frequently play the role of forbidden patterns in this paper.}\label{fig-pats}
\end{figure}

Throughout this paper, whenever we deal with the enumeration of relational
structures, we treat them as unlabeled objects, that is, we consider an entire
isomorphism class as a single object.

We assume the reader is familiar with the concept of Hasse diagram of a
partial order. We shall be mainly interested in partial orders that avoid some
of the three orders \tpt, \tpo, and \N, depicted on Figure~\ref{fig-pats}.
These classes of partial orders have been studied before. The \tptp-free posets
are known as \emph{interval orders}, and they are the prototypical example
of a class of combinatorial structures enumerated by Fishburn numbers. The
\N-free posets\footnote{Beware that some authors (e.g. Khamis~\cite{Khamis})
give the term `\N-free poset' a meaning subtly different from ours.} are also
known as \emph{series-parallel posets}. They are exactly the
posets that can be constructed from a single-element poset by a repeated
application of direct sums and disjoint unions. The \tpop-free
posets do not seem to admit such a simple characterization, and their structural
description is the topic of ongoing
research~\cite{Skandera,SkanderaReed,LewisZhang,GPMR}.

The \tpto-free posets, i.e., the posets avoiding both \tpt\ and \tpo, are also
known as \emph{semiorders}, and have been introduced by Luce~\cite{Luce} in the
1950s. They are enumerated by Catalan numbers~\cite{KimRoush}. The \tptn-free
posets are enumerated by Catalan numbers as well, as shown by Disanto et
al.~\cite{DisantoAAM,DPPR13}.

\subsection{Interval orders and Fishburn matrices}\label{ssec-matrices}
Let us briefly summarize several known facts about interval orders and their
representations. A detailed treatment of this topic appears, e.g., in Fishburn's
book~\cite{fishburn1985interval}.

Let $R=(X,\cl)$ be a poset. An \emph{interval representation} of $R$ is a
mapping $I$ that associates to every element $x\in X$ a closed
interval $I(x)=[l_x,r_x]$ in such a way that for any two elements
$x,y\in X$, we have $x\cl y$ if and only if $r_x<l_y$. Note that
we allow the intervals $I(x)$ to be degenerate, i.e., to consist of a single
point.

A poset has an interval representation if and only if it is \tptp-free, i.e., if
it is an interval order. An interval representation $I$ is \emph{minimal}, if it
satisfies these conditions:
\begin{itemize}
 \item for every $x\in X$, the endpoints of $I(x)$ are positive integers
 \item there is a positive integer 
$m\in\bbN$ such that for every
$k\in\{1,2,\dotsc,m\}$ there is an interval $I(x)$ whose right endpoint is $k$,
as well as an interval $I(y)$ whose left endpoint is~$k$. 
\end{itemize}

Each interval order $R$ has a unique minimal interval representation. The
integer $m$, which corresponds to the number of distinct endpoints in the
minimal representation, is known as \emph{the magnitude} of~$R$. Note that
$x\in X$ is a minimal element of $R$ if and only if the left endpoint of
$I(x)$ is equal to 1, and $x$ is a maximal element if and only if the right
endpoint of $I(x)$ is equal to~$m$.

Two elements $x,y$ of a poset $R=(X,\cl)$ are \emph{indistinguishable} if
for every element $z\in X$ we have the equivalences $x\cl z\iff y\cl z$ and
$z\cl x\iff z\cl y$. In an interval order $R$ with a minimal interval
representation~$I$, two elements $x,y$ are indistinguishable if and only if
$I(x)=I(y)$. An interval order is \emph{primitive} if it has no two
indistinguishable elements. Every interval order $R$ can be uniquely
constructed from a primitive interval order $R'$ by replacing each element
$x\in R'$ by a group of indistinguishable elements containing~$x$.

Given a matrix $M$, we use the term \emph{cell $(i,j)$ of $M$} to refer to the
entry in the $i$-th row and $j$-th column of $M$, and we let
$M_{i,j}$ denote its value. We assume that the rows of a matrix are numbered top
to bottom, that is, the top row has number 1. The \emph{weight} of a matrix $M$
is the sum of its entries\footnote{Some earlier papers use the term \emph{size
of $M$} instead of weight of~$M$. However, in other contexts the term `size'
often refers to the number of rows of a matrix. We therefore prefer to use the
less ambiguous term `weight' in this paper.}. Similarly, the weight of a row (or
a column, or a diagonal) of a matrix is the sum of the entries in this row (or
column or diagonal).

\begin{figure}
 \hfil\includegraphics[width=0.7\textwidth]{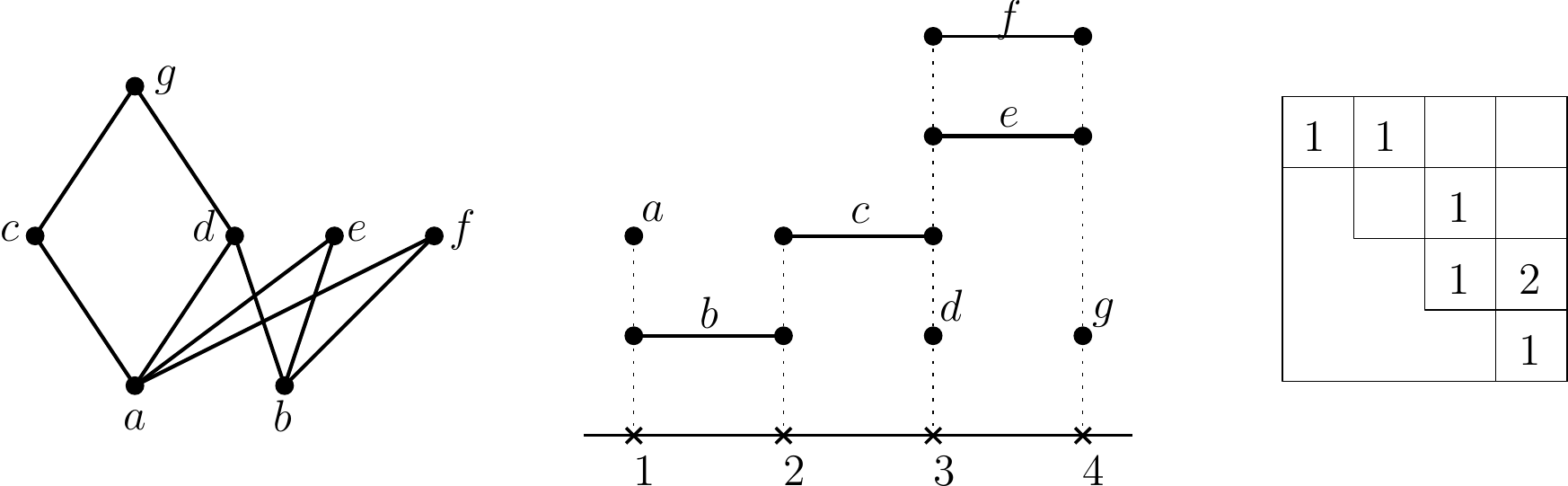}
\caption{Three Fishburn structures: an interval order, its minimal interval
representation, and its Fishburn matrix. In this paper,
we use the convention that cells of value 0 in a matrix are depicted as empty
boxes.}\label{fig-fish}
\end{figure}

A \emph{Fishburn matrix} is an upper-triangular square matrix of nonnegative
integers whose every row and every column has nonzero weight. In other words, an
$m\times m$ matrix $M$ of nonnegative integers is a Fishburn matrix if it
satisfies these conditions:
\begin{itemize}
 \item $M_{i,j}=0$ whenever $i>j$,
 \item for every $i\in\{1,\dotsc,m\}$ there is a $j$ such that $M_{i,j}>0$, and
 \item for every $j\in\{1,\dotsc,m\}$ there is an $i$ such that $M_{i,j}>0$.
\end{itemize}
Fishburn matrices were introduced by Fishburn~\cite{fishburn1985interval} as a
convenient way to represent interval orders. Given an interval order $R=(X,\cl)$
of magnitude $m$ with minimal representation $I$, we represent it by an $m\times
m$ matrix $M$ such that $M_{i,j}$ is equal to $|\{x\in X;\; I(x)=[i,j]\}|$. 
We then say that an element $x\in X$ is \emph{represented} by the cell $(i,j)$
of $M$, if $I(x)=[i,j]$. Thus, every element of $R$ is represented by a unique
cell of $M$, and $M_{i,j}$ is the number of elements of $R$ represented by cell
$(i,j)$ of $M$. This correspondence yields a bijection between Fishburn matrices
and interval orders. For an element $x\in X$, we let $c_x$ denote the cell of
$M$ representing $x$. Conversely, for a cell $c$ of $M$, we let $X_c$ be the set
of elements of $X$ represented by~$c$. Note that the value in cell $c$ is
precisely the cardinality of~$X_c$.

In the rest of this paper, we will use Fishburn matrices as our main family of
Fishburn-enumerated object. We now show how the basic features of interval
orders translate into matrix terminology.
\begin{observation}\label{obs-fish} Let $R=(X,\cl)$ be an interval order, and
let $M$ be the corresponding Fishburn matrix. Let $x$ and $x'$ be two elements
of $R$, represented respectively by cells $(i,j)$ and $(i',j')$ of $M$. The
following
holds:
\begin{itemize}
\item The size of $R$ is equal to the weight of $M$.
\item The number of minimal elements of $R$ is equal to the weight of the
first row of $M$, and the number of maximal elements of $R$ is equal to the
weight of the last column of $M$.
\item The elements $x$ and $x'$ are indistinguishable in $R$ if and only if
they are represented by the same cell of~$M$, i.e., $i=i'$ and $j=j'$.
\item We have $x'\cl x$ if and only if $j'<i$.
\end{itemize}
\end{observation}

\begin{figure}
\hfil\includegraphics[width=0.6\textwidth]{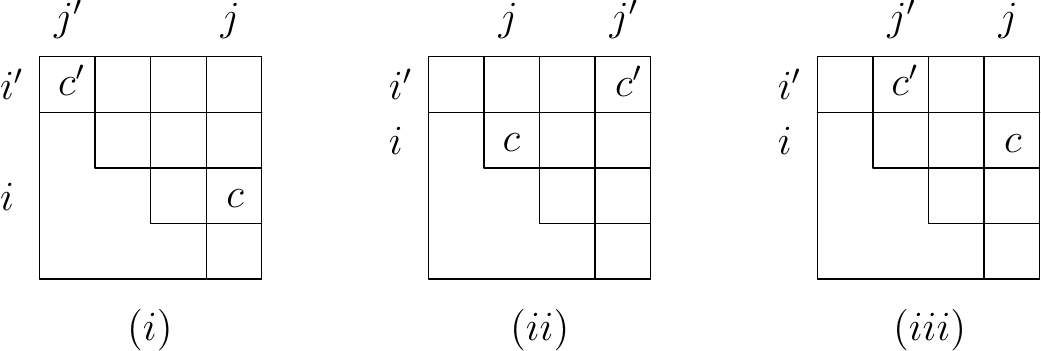}
\caption{Mutual positions of two cells $c$ and $c'$ in a Fishburn matrix: (i)
 $c$ is greater than $c'$, (ii) $c$ is strictly South-West of $c'$, and (iii)
$c$ is strictly South-East of~$c'$.} \label{cellpos}
\end{figure}

Let $M$ be a Fishburn matrix, and let $c=(i,j)$ and $c'=(i',j')$ be two cells 
of $M$ such that $i\le j$ and $i'\le j'$, i.e., the two cells are on or above 
the main diagonal. We shall frequently use the following terminology (see
Figure~\ref{cellpos}):
\begin{itemize}
 \item Cell $c$ is \emph{greater} than cell $c'$ (and $c'$ is \emph{smaller} 
than $c$), if $j'<i$. The cells $c$ and $c'$ are \emph{incomparable} if 
neither of them is greater than the other. These terms are motivated 
by the last part of Observation~\ref{obs-fish}. Note that two cells $c$ and $c'$
of $M$ are comparable if and only if the smallest rectangle containing both $c$
and $c'$ also contains at least one cell strictly below the main diagonal
of~$M$. 
\item Cell $c$ is \emph{South} of $c'$, if $i>i'$ and $j=j'$. North, West and 
East are defined analogously. Note that in all these cases, the two cells $c$ 
and $c'$ are incomparable.
\item Cell $c$ is \emph{strictly South-West} (or strictly SW) from $c'$ (and
$c'$ is strictly NE from $c$) if $i>i'$ and $j<j'$. This again implies that the
two cells are incomparable.
\item Cell $c$ is \emph{strictly South-East} (or strictly SE) from $c'$ (and
$c'$ is strictly NW from $c$) if $c$ and $c'$ are incomparable, and moreover
$i>i'$ and $j>j'$.
\item Cell $c$ is \emph{weakly SW} from $c'$ is $c$ is South, West of strictly
SW of $c'$. Weakly NE, weakly NW and weakly SE are defined analogously.
\end{itemize}

With this terminology, \tpop-avoidance and \N-avoidance of interval orders may 
be characterized in terms of Fishburn matrices, as shown by the next lemma,
whose first part essentially already appears in the work of Dukes et
al.~\cite{DJK}. Refer to Figure~\ref{fig-swse}.

\begin{lemma}\label{lem-swse}
Let $R=(X,\cl)$ be an interval order represented by a Fishburn matrix~$M$. 
Let $x$ and $y$ be two elements of $X$, represented by cells $c_x$ and $c_y$
of~$M$. 
\begin{enumerate}
\item
The cell $c_x$ is strictly SW of $c_y$ if and only if $X$ contains two
elements $u$ and $v$ which together with $x$ and $y$ induce a
copy of \tpo\ in $R$ such that $u\cl x \cl v$, and $y$ is incomparable to
each of $u,x,v$.
\item
The cell $c_x$ is strictly NW of $c_y$ if and only if $X$ contains two elements
$u$ and $v$ which together with $x$ and $y$ induce a copy of \N\ in
$R$ such that $u\cl y$, $u\cl v$ and $x\cl v$, and the remaining pairs among
$u,v,x,y$ are incomparable.
\end{enumerate}

In particular, $R$ is \tpop-free if and only if the matrix $M$ has
no two distinct nonzero cells in a strictly SW position, and 
$R$ is \N-free if and only if $M$ has no two distinct nonzero cells
in a strictly NW position.
\end{lemma}

\begin{figure}
 \hfil\includegraphics[width=\textwidth]{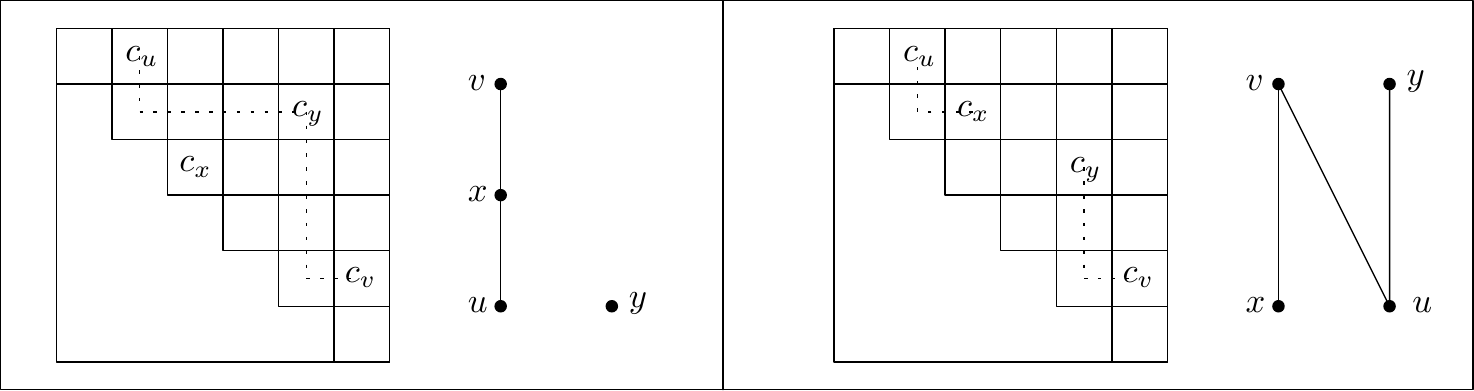}
\caption{The two parts of Lemma~\ref{lem-swse}. Left: if $c_x$ is strictly SW
of $c_y$, we get an occurrence of \tpo. Right: if $c_x$ is strictly NW
of $c_y$, we get an an occurrence of \N.}\label{fig-swse}
\end{figure}

\begin{proof}
We will only prove the second part of the lemma, dealing with \N-avoidance. The 
first part can be proved by a similar argument, as shown by Dukes at
al.~\cite[proof of Proposition 16]{DJK}. 

Suppose first that $R$ contains a copy of \N\ induced by four elements $u$,
$v$, $x$ and $y$, which form exactly three comparable pairs $u\cl y$, $u\cl v$
and $x\cl v$. Let $c_u=(i_u,j_u)$, $c_v=(i_v,j_v)$, $c_x=(i_x,j_x)$ and
$c_y=(i_y,j_y)$ be the cells of $M$ representing these four elements. We claim
that $c_x$ is strictly NW from~$c_y$. To see this, note that $c_x$ is smaller
than $c_v$ while $c_y$ is not smaller than $c_v$, hence $j_x<i_v\le j_y$.
Similarly, $c_y$ is greater than $c_u$ while $c_x$ is not, implying that $i_x\le
j_u<i_y$. Since $c_x$ and $c_y$ are incomparable, it follows that $c_x$ is
strictly NW from~$c_y$.

Conversely, suppose that $M$ has two nonzero cells $c_x=(i_x,j_x)$ and 
$c_y=(i_y,j_y)$, with $c_x$ being strictly NW from~$c_y$. We thus have the
inequalities $i_x<i_y\le j_x<j_y$, where the inequality $i_y\le j_x$ follows
from the fact that $c_x$ and $c_y$ are incomparable. Let $c_u$ be any nonzero
cell in column~$i_x$. This choice guarantees that $c_u$ is incomparable to $c_x$
and smaller than~$c_y$. Similarly, let $c_v$ be a nonzero cell in row~$j_y$.
Then $c_v$ is incomparable to $c_y$ and greater than both $c_x$ and~$c_u$. Any
four elements $x$, $y$, $u$ and $v$ of $R$ represented by the four cells
$c_x$, $c_y$, $c_u$ and $c_v$ induce a copy of~\N.
\end{proof}

\section{Catalan pairs}\label{sec-catpairs}

We begin by defining the concept of Catalan pairs of type 1, originally
introduced by Disanto et al.~\cite{DisantoAAM}, who called them simply `Catalan
pairs'.
\begin{definition}
A \emph{Catalan pair of type 1} (or \emph{C1-pair} for short) is a relational 
structure $(S,R)$ on a finite set $X$ satisfying the following axioms.
\begin{itemize}
 \item[C1a)] $S$ and $R$ are both partial orders on~$X$.
 \item[C1b)] Any two distinct elements of $X$ are comparable by exactly one of
the orders $S$ and $R$.
 \item[C1c)] For any three distinct elements $x,y,z$ satisfying $xSy$ and
$yRz$, we have $xRz$.
\end{itemize}
We let $\Co_n$ denote the set of unlabeled C1-pairs on $n$ vertices.
\end{definition}

Note that axiom C1c might be replaced by the seemingly weaker condition that
$X$ does not contain three distinct elements $x$, $y$ and $z$ satisfying
$xSy$, $yRz$ and $xSz$ (see the left part of Figure~\ref{fig-forb}). Indeed, if
$xSy$ and $yRz$ holds for some $x$, $y$ and $z$, then both $zRx$ and $zSx$ would
contradict axiom C1a or C1b, so the role of C1c is merely to exclude the
possibility $xSz$, leaving $xRz$ as the only option. This also shows that the
definition of C1-pairs would not be affected if we replaced C1c by the following
axiom, which we denote C1c*: ``For any three distinct elements $x,y,z$
satisfying $xSy$ and $zRy$, we have $zRx$.''

We remark that it is easy to check that in a C1-pair $(S,R)$ on a vertex set
$X$, the relation $S\cup R$ is a linear order on~$X$.

As shown by Disanto et al.~\cite{DisantoAAM},  if $(S,R)$ is a C1-pair, then 
$R$ is a \tptn-free poset, and conversely, for any \tptn-free poset 
$R$, there is, up to isomorphism, a unique C1-pair $(S,R)$. Since there are, up
to isomorphism, $C_n$ distinct \tptn-free posets on an $n$-elements set,
this implies that C1-pairs are enumerated by Catalan numbers. In fact, the next
lemma shows that the first component of a C1-pair $(S,R)$ can be easily
reconstructed from the Fishburn matrix representing the second component.

\begin{figure}
 \hfil\includegraphics[width=0.6\textwidth]{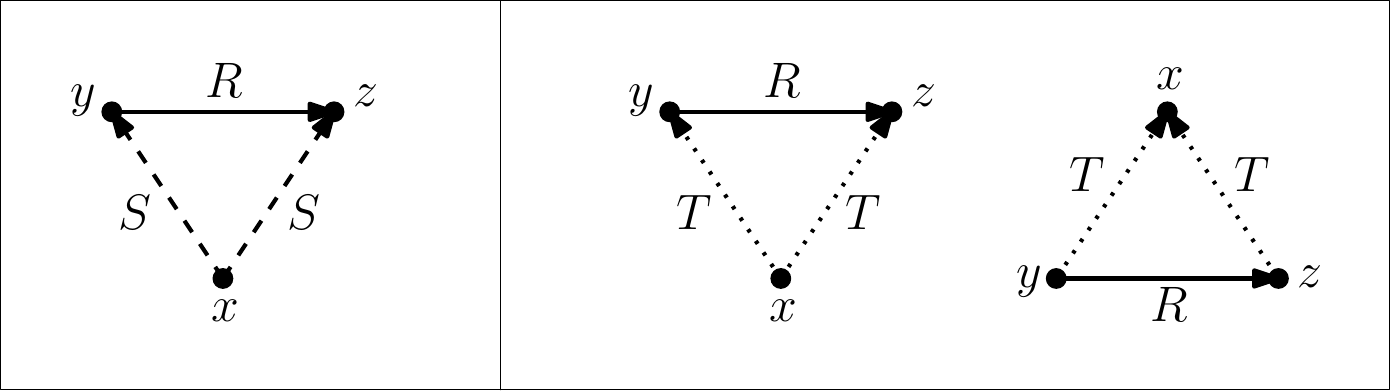}
\caption{Left: the substructure forbidden in C1-pairs by axiom C1c. Right: the
two substructures forbidden in C2-pairs by axiom C2c.}\label{fig-forb} 
\end{figure}

\begin{lemma}\label{lem-c1mat} Let $(S,R)$ be a C1-pair, and let $M$ be the 
Fishburn matrix representing the poset~$R$. Then $S$ has the following 
properties:
\begin{itemize}
\item[(a)] If $c$ is a nonzero cell of $M$ and $X_c$ the set of elements
represented by $c$, then $X_c$ is a chain in~$S$, that is, the restriction of 
$S$ to $X_c$ is a linear order.
\item[(b)] If $c$ and $d$ are distinct nonzero cells of $M$, representing sets
of elements $X_c$ and $X_d$ respectively, and if $c$ is weakly SW from $d$,
then for any $x\in X_c$ and $y\in X_d$ we have $xSy$. 
\item[(c)] Apart from the situations described in parts (a) and (b), no other 
pair $(x,y)\in X^2$ is $S$-comparable. 
\end{itemize}
\end{lemma}
\begin{proof}
Let $(S,R)$ be a C1-pair. Property (a) of the Lemma follows directly from the
fact that the elements of $X_c$ form an antichain in $R$. To prove property
(b), fix distinct nonzero cells $c=(i_c,j_c)$ and $d=(i_d,j_d)$ such that $c$
is weakly SW from~$d$, and choose $x\in X_c$ and $y\in X_d$ arbitrarily. Since
$x$ and $y$ are $R$-incomparable by Observation~\ref{obs-fish}, we must have
either $xSy$ or $ySx$. Suppose for contradiction that $ySx$ holds. As $c$ is
weakly SW from $d$, we know that $i_c>i_d$ or $j_c<j_d$. Suppose that $i_c>i_d$,
as the case $j_c<j_d$ is
analogous. Let $e$ be any nonzero cell in column $i_d$ of $M$, and let $z\in X$
be an element represented by~$e$. Note that $e$ is smaller than $c$ and
incomparable to~$d$. It follows that $y$ and $z$ are comparable by $S$.
However, if $zSy$ holds, then we get a contradiction with the transitivity of
$S$, due to $ySx$ and $zRx$. On the other hand, if $ySz$ holds, we see that
$x$, $y$ and $z$ induce the substructure forbidden by axiom C1c.

To see that property (c) holds, note that by Lemma~\ref{lem-swse}, the matrix 
$M$ has no two nonzero cells in strictly NW position. Thus, if $x$ and $y$ are 
distinct elements of $X$ represented by cells $c_x$ and $c_y$, then either 
$c_x=c_y$ and $x,y$ are $S$-comparable by property (a), or $c_x$ and $c_y$ are 
in a weakly SW position and $x,y$ are $S$-comparable by (b), or one of the two 
cells is smaller than the other, which means that the two elements are 
$R$-comparable.
\end{proof}

Various Catalan-enumerated objects, such as noncrossing matchings, Dyck paths,
or 132-avoiding permutations, can be encoded in a natural way as C1-pairs 
(see~\cite{DisantoAAM}). However, there are also examples of Catalan objects, 
such as nonnesting matchings or \tpto-free posets, possessing a natural
underlying structure that satisfies a different set of axioms. This motivates
our next definition.

\begin{definition} A \emph{Catalan pair of type 2} (or \emph{C2-pair} for short)
is a relational structure $(T,R)$ on a finite set $X$ with the following
properties.
\begin{itemize}
 \item[C2a)] $R$ and $T\cup R$ are both partial orders on~$X$.
 \item[C2b)] Any two distinct elements of $X$ are comparable by exactly one of
the relations $T$ and~$R$.
 \item[C2c)] There are no three distinct elements $x,y,z\in X$ satisfying $xTy$,
$xTz$ and $yRz$, and also no three elements $x,y,z\in X$ satisfying $yTx$,
$zTx$ and $yRz$.
\end{itemize}
We let $\Ct_n$ denote the set of unlabeled C2-pairs on $n$ vertices.
\end{definition}

\begin{lemma}\label{lem-C2} If $(T,R)$ is a C2-pair on a vertex set $X$, then
$R$ is a \tpto-free poset. 
\end{lemma}
\begin{proof}
Let $(T,R)$ be a C2-pair. Suppose for contradiction that $R$ contains a copy of 
\tpt\ induced by four elements $\{x,y,u,v\}$, with $xRy$ and $uRv$. Any two 
$R$-incomparable elements must be comparable in $T$,
so we may assume, without loss of generality, that $xTu$ holds. Then $xTv$
holds by transitivity of $T\cup R$, contradicting axiom C2c.
Similar reasoning shows that $R$ is \tpop-free.
\end{proof}

As in the case of C1-pairs, a C2-pair $(T,R)$ is uniquely determined by its
second component, and its first component can be easily reconstructed from the
Fishburn matrix representing the second component. 

\begin{lemma}\label{lem-C2enum}
Let $R$ be \tpto-free poset represented by a Fishburn matrix~$M$. Let $T$ 
be a relation satisfying these properties:
\begin{itemize}
\item[(a)] If $c$ is a nonzero cell of $M$ and $X_c$ the set of elements
represented by $c$, then $X_c$ is a chain in~$T$, that is, the restriction of 
$T$ to $X_c$ is a linear order.
\item[(b)] If $c$ and $d$ are distinct nonzero cells of $M$, representing sets
of elements $X_c$ and $X_d$ respectively, and if $c$ is weakly NW from $d$,
then for any $x\in X_c$ and $y\in X_d$ we have $xTy$. 
\item[(c)] Apart from the situations described in parts (a) and (b), no other 
pair $(x,y)\in X^2$ is $T$-comparable.
\end{itemize}
Then $(T,R)$ is a C2-pair. Moreover, if $(T',R)$ is another C2-pair, then
$(T,R)$ and $(T',R)$ are isomorphic relational structures. It follows that
on an $n$-element vertex set $X$, there are $C_n$ isomorphism types of Catalan
pairs, where $C_n$ is the $n$-th Catalan number. 
\end{lemma}
\begin{proof}
Let $R$ be a \tpto-free partial order on a ground set~$X$, 
represented by a Fishburn matrix $M$. Let $T$ be a relation satisfying 
properties (a), (b) and (c) of the lemma. 

We claim that $(T,R)$ is a C2-pair. It is easy to see that $(T,R)$ satisfies
axiom C2a. Let us verify that C2b holds as well, i.e., any two distinct
elements $x, y\in X$ are comparable by exactly one of $T$, $R$. Since clearly
no two elements can be simultaneously $T$-comparable and $R$-comparable, it is
enough to show that any two $R$-incomparable elements $x,y\in X$ are
$T$-comparable. If $x$ and $y$ are also $R$-indistinguishable, then they are
represented by the same cell $c$ of~$M$, and they are $T$-comparable by (a). If
$x$ and $y$ are $R$-distinguishable, then they are represented by two distinct
cells, say $c_x$ and $c_y$. By Lemma~\ref{lem-swse}, $M$ has no pair of nonzero
cells in strictly SW position, and hence $c_x$ and $c_y$ must be in weakly NW
position. Hence $x$ and $y$ are $T$-comparable by (b). 

To verify axiom C2c, assume first, for contradiction, that there are three
elements $x,y,z\in X$ satisfying $xTy$, $xTz$ and $yRz$. Let $c_x=(i_x,j_x)$,
$c_y=(i_y,j_y)$ and $c_z=(i_z,j_z)$ be the cells of $M$ representing $x$, $y$
and $z$, respectively. Since $c_y$ is smaller than $c_z$, we see that $j_y<i_z$.
Since $(x,y)$ belongs to $T$, we see that either $c_x=c_y$ or $c_x$ is weakly
NW from $c_y$. In any case, $c_x$ is smaller than $c_z$, contradicting $xTz$.
An analogous argument shows that there can be no three elements $x,y,z$
satisfying $yTx$, $zTx$ and~$yRz$. We conclude that $(T,R)$ is a C2-pair.

It remains to argue that for the \tpto-free poset $R$, there is, up to
isomorphism, at most one relation $T$ forming a C2-pair with~$R$. 
Fix a relation $T'$ such that $(T',R)$ is a C2-pair. We will first show that
$T'$ satisfies the properties (a), (b) and (c) of the lemma.
Property (a) follows directly from axioms C2a and~C2b. 

To verify property (b), fix elements $x,y\in X$ represented by distinct cells
$c_x=(i_x,j_x)$ and $c_y=(i_y,j_y)$ such that $c_x$ is weakly NW from $c_y$. We
claim that $(x,y)\in T'$. Suppose for contradiction that this is not the
case. Since $x$ and $y$ are $R$-incomparable, this means that $(y,x)\in T'$
by axiom C2b. Since $c_x$ is weakly NW from $c_y$, we know that $i_x<i_y$ or
$j_x<j_y$. Suppose that $i_x<i_y$ (the case $j_x<j_y$ is analogous). Let $c_z$
be any nonzero cell in column~$i_x$ and $z\in X$ an element represented
by~$c_z$. This choice of $z$ guarantees that $(z,y)\in R$ while $z$ and $x$ are
$R$-incomparable. From $zRy$ and $yT'x$ we deduce, by the transitivity of $R\cup
T'$ and the $R$-incomparability of $x$ and $z$, that $(z,x)$ belongs to $T'$,
which contradicts axiom C2c. The fact that $T'$ satisfies (c) follows from
axiom C2b and the fact that no two nonzero cells of $M$ are in strictly SW
position. 

We conclude that when  $(T,R)$ and $(T',R)$ are C2-pairs, the relations $T$
and $T'$ may only differ by prescribing different linear orders on the classes
of $R$-indistinguishable elements. This easily implies that $(T,R)$ and
$(T',R)$ are isomorphic relational structures.
\end{proof}

\subsection{Statistics on Catalan pairs}\label{ssec-catstats}
We have seen that Catalan pairs naturally encode the structure of Fishburn
matrices representing \tptn-free and \tpto-free posets. However, to
exploit known combinatorial properties of Catalan-enumerated objects, it is
convenient to relate Catalan pairs to more familiar Catalan objects. Our Catalan
objects of choice are the Dyck paths. 

\begin{figure}
 \hfil\includegraphics{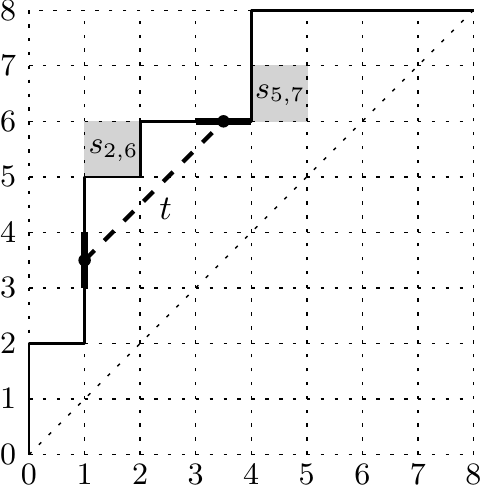}
\caption{A Dyck path (solid) with a tunnel $t$ (dashed). The two steps 
associated to $t$ are highlighted in bold. The unit square $s_{5,7}$ is
below the path while $s_{2,6}$ is above it.}\label{fig-tunnel}
\end{figure}

\begin{definition}
 A \emph{Dyck path} of order $n$ is a lattice path $P$ joining the point $(0,0)$
with the point $(n,n)$, consisting of $n$ \emph{up-steps} and $n$
\emph{right-steps}, where an up-step joins a point $(i,j)$ to $(i,j+1)$ and a
right-step joins $(i,j)$ to $(i+1,j)$, and moreover, every point $(i,j)\in P$
satisfies $i\le j$. We let $\D_n$ denote the set of Dyck paths of order~$n$.

For a Dyck path $P$, we say that a step $s$ of $P$ \emph{precedes} a step $s'$
of $P$ if $s$ appears on $P$ before $s'$ when we follow $P$ from $(0,0)$
to~$(n,n)$.

Given a Dyck path $P$, a \emph{tunnel} of $P$ is a segment $t$ 
parallel to the diagonal line of equation $y=x$, such that the bottom-left
endpoint of $t$ is in the middle of an up-step of $P$, the top-right endpoint
is in the middle of a right-step of $P$, and all the internal points of $t$ are
strictly below the path~$P$. See Figure~\ref{fig-tunnel}. We refer to the
up-step and the right-step that contain the two endpoints of $t$ as \emph{the
steps associated to~$t$}.
\end{definition}

Note that a Dyck path of order $n$ has exactly $n$ tunnels, and every
step of the path is associated to a unique tunnel.  

Let $t_1$ and $t_2$ be two distinct tunnels of a Dyck path $P$. Let $u_i$ and
$r_i$ be the up-step and right-step associated to $t_i$, respectively. Suppose
that $u_1$ precedes $u_2$ on the path~$P$. If $r_2$ precedes $r_1$ on $P$, we
say that $t_2$ is \emph{nested} within~$t_1$. If $r_1$ precedes $r_2$, then we
easily see that $r_1$ also precedes $u_2$ on~$P$. In such case, we say that
$t_1$ \emph{precedes}~$t_2$.

The following construction, due to Disanto et al.~\cite{DisantoAAM},
shows how to represent a Dyck path by a C1-pair. 

\begin{fact}[\cite{DisantoAAM}]\label{fac-dc1} Let $P$ be a Dyck path and let
$X$ be the set of tunnels of $P$. Define a relational structure $\co(P)=(S,R)$
on $X$ as follows: for two distinct tunnels $t_1$ and $t_2$, put $(t_1,t_2)\in
S$ if and only if $t_1$ is nested within $t_2$, and $(t_1,t_2)\in R$ if and only
if $t_1$ precedes $t_2$. Then $\co(P)=(S,R)$ is a C1-pair, and this
construction yields a bijection between Dyck paths of order $n$ and isomorphism
types of C1-pairs of order~$n$. 
\end{fact}

There is also a simple way to encode a Dyck path by a C2-pair, which we now
describe. Let $P$ be a Dyck path of order $n$ and let
$(i,j)\in\{1,\dotsc,n\}^2$ be a lattice point. Let $s_{i,j}$ be the axis-aligned
unit square whose top-right corner is the point $(i,j)$ (see
Figure~\ref{fig-tunnel}). We say that
$s_{i,j}$ is \emph{above} the Dyck path $P$ if the interior of $s_{i,j}$ is
above $P$ (with the boundary of $s_{i,j}$ possibly overlapping with $P$).
Notice that $s_{i,j}$ is above $P$ if and only if the $i$-th right-step of $P$
is preceded by at most $j$ up-steps. If
$s_{i,j}$  is not above $P$, we say that it is \emph{below}~$P$; see
Figure~\ref{fig-tunnel}.

\begin{lemma}\label{lem-dc2}
Let $P$ be a Dyck path of order $n$, and let $X$ be the set $\{1,2,\dotsc,n\}$.
Define a relational structure $\ct(P)=(T,R)$ on $X$ as follows: for any
$i,j\in X$, put $(i,j)\in T$ if and only if $i<j$ and $s_{i,j}$ is below $P$,
and put $(i,j)\in R$ if and only if $i<j$ and $s_{i,j}$ is above $P$. Then 
$\ct(P)$ is a C2-pair, and this construction is a bijection between
Dyck paths of order $n$ and isomorphism types of C2-pairs.
\end{lemma}
\begin{proof}
It is routine to verify that $\ct(P)$ satisfies the axioms of C2-pairs, and
that distinct Dyck paths map to distinct C2-pairs. Since both Dyck paths of
order $n$ and C2-pairs of order $n$ are enumerated by Catalan numbers, the
mapping is indeed a bijection. 
%
\end{proof}

We will now focus on combinatorial statistics on Catalan objects. For our
purposes, a \emph{statistic} on a set $\cA$ is any nonnegative integer function
$f\colon \cA\to\bbN_0$. For an integer $k$ and a statistic $f$, we use the
notation $\cA[f=k]$ as shorthand for $\{x\in\cA: f(x)=k\}$. This notation
extends naturally to two or more statistics, e.g., if $g$ is another statistic
on $\cA$, then $\cA[f=k, g=\ell]$ denotes the set $\{x\in\cA: f(x)=k,
g(x)=\ell\}$. Two statistics $f$ and $g$ are \emph{equidistributed} (or
\emph{have the same distribution}) on $\cA$, if $\cA[f=k]$ and $\cA[g=k]$ have
the same cardinality for every~$k$. The statistics $f$ and $g$ have
\emph{symmetric joint distribution} on $\cA$, if $\cA[f=k, g=\ell]$ and
$\cA[f=\ell, g=k]$ have the same cardinality for every $k$ and~$\ell$. Clearly,
if $f$ and $g$ have symmetric joint distribution, they are equidistributed.

Recall that for a binary relation $R$ on a set $X$, an
element $x\in X$ is minimal, if there is no $y\in X\setminus\{x\}$ such that
$(y,x)$ is in $R$. Recall also that $\MMin{R}$ is the set of minimal elements of
$R$, and $\mmin{R}$ is the cardinality of~$\MMin{R}$. Similarly, $\MMax{R}$
is the set of maximal elements of $R$, and $\mmax{R}$ its cardinality.

Let $P$ be a Dyck path of order $n$, and let $\co(P)=(S,R)$ be the
corresponding C1-pair, as defined in Fact~\ref{fac-dc1}. We will be interested
in the four statistics $\mmin{S}$, $\mmax{S}$, $\mmin{R}$ and $\mmax{R}$. It
turns out that these statistics correspond to well known statistics of Dyck
paths. To describe the correspondence, we need more terminology.

For a Dyck path $P$, the \emph{initial ascent} is the maximal sequence of
up-steps preceding the first right-step, and the \emph{final descent} is the
maximal sequence of right-steps following the last up-step. Let $\asc(P)$
and $\des(P)$ denote the length of the initial ascent and final descent of $P$,
respectively. A \emph{return} of $P$ is a right-step whose right endpoint
touches the diagonal line $y=x$. Let $\ret(P)$ denote the number of returns
of~$P$. Finally, a \emph{peak} of $P$ is an up-step of $P$ that is immediately
followed by a right-step, and $\pea(P)$ denotes the number of peaks
of~$P$.

\begin{observation}\label{obs-stat-c1} Let $P$ be a Dyck path and let
$(S,R)=\co(P)$ be the corresponding C1-pair.
\begin{itemize}
 \item A tunnel $t$ of $P$ is minimal in $R$ if and only if the up-step
associated to $t$ precedes the first right-step of $P$. In particular 
$\mmin{R}=\asc(P)$. Symmetrically, $t$ is maximal in $R$ if and only if its
associated right-step succeeds the last up-step of $P$. Hence,
$\mmax{R}=\des(P)$.
\item A tunnel $t$ of $P$ is maximal in $S$ if and only if its associated
right-step is a return of $P$. Hence, $\mmax{S}=\ret(P)$.
\item A tunnel $t$ of $P$ is minimal in $S$ if and only if its associated
up-step is immediately succeeded by its associated right-step. Hence,
$\mmin{S}=\pea(P)$.
\end{itemize}
\end{observation}

Suppose now that $(T,R)=\ct(P)$ is a C2-pair representing a Dyck path $P$ by
the bijection of Lemma~\ref{lem-dc2}. We again focus on the statistics
$\mmin{T}$, $\mmax{T}$, $\mmin{R}$ and $\mmax{R}$. It turns out that they
again correspond to the Dyck path statistics introduced above.

\begin{observation}\label{obs-stat-c2}
Let $P$ be a Dyck path of order $n$ and let $(T,R)=\ct(P)$ be the corresponding
C2-pair.
\begin{itemize}
 \item An element $i\in\{1,2,\dotsc,n\}$ is minimal in $R$ if and only if the
unit square $s_{1,i}$ is below $P$, which happens if and only if $\asc(P)\ge
i$. In particular, $\asc(P)=\mmin{R}$. Symmetrically, $\des(P)=\mmax{R}$.
\item An element $i\in\{1,2,\dotsc,n\}$ is minimal in $T$ if and only if for
every $j<i$, the unit square $s_{j,i}$ is above $P$, which is if and only if
the point $(i-1,i-1)$ belongs to $P$. It follows that $\mmin{T}=\ret(P)$.
Symmetrically, $i\in\{1,2,\dotsc,n\}$ is maximal in $T$ if and only if for
every $j>i$ the square $s_{i,j}$ is above $P$, which is if and only if the
point $(i,i)$ belongs to~$P$. It follows that $\mmax{T}=\ret(P)$.
\end{itemize}
\end{observation}

The statistics $\asc$, $\des$, $\ret$ and $\pea$ are all very well studied. We
now collect some known facts about them. 

\begin{fact}\label{fac-catstat} Let $\D_n$ be the set of Dyck paths of
order~$n$.
\begin{itemize}
 \item The sets $\D_n[\asc=k]$, $\D_n[\des=k]$ and $\D_n[\ret=k]$ all have
cardinality $\frac{k}{2n-k}\binom{2n-k}{n}$. See \cite[A033184]{oeis}. Among the 
three statistics $\asc$, $\des$ and $\ret$, any two have 
symmetric joint distribution on~$\D_n$.
 \item The set $\D_n[\pea=k]$ has cardinality
$N(n,k)=\frac{1}{k}\binom{n-1}{k-1}\binom{n}{k-1}$. The numbers $N(n,k)$ are
known as Narayana numbers \cite[A001263]{oeis}. The Narayana numbers satisfy 
$N(n,k)=N(n,n-k+1)$, and hence $\D_n[\pea=k]$ has the same cardinality as 
$\D_n[\pea=n-k+1]$.
 \item In fact, for any $n$, $k$, $\ell$, and $m$, the set
 $\D_n[\asc=k,\ret=\ell,\pea=m]$ has the same cardinality as the set 
$\D_n[\asc=\ell,\ret=k,\pea=n-m+1]$. This follows, e.g., from an
involution on Dyck paths constructed by Deutsch~\cite{DeutschInvol}.
\end{itemize}
\end{fact}
For future reference, we rephrase some of these facts in the terminology of
Catalan pairs. Recall that $\Co_n$ and $\Ct_n$ denote respectively the set of
C1-pairs and the set of C2-pairs of order~$n$.

\begin{proposition}\label{pro-cat1}
For any $n\ge 0$, there is a bijection $\psi\colon\Co_n\to\Ct_n$ with the
following properties. Let $(S,R)\in\Co_n$ be a C1-pair, and let
$(T',R')\in\Ct_n$ be its image under~$\psi$. Then
\begin{itemize}
 \item $\mmax{S}=\mmax{T'}$, and 
 \item $\mmax{R}=\mmax{R'}$.
\end{itemize}
\end{proposition}
\begin{proof}
Given a C1-pair $(S,R)\in\Co_n$, fix the Dyck path $P\in\D_n$ satisfying
$(S,R)=\co(P)$, and define the C2-pair $(T',R')=\ct(P)$. The mapping
$(S,R)\mapsto(T',R')$ is the bijection~$\psi$. By Observations~\ref{obs-stat-c1}
and \ref{obs-stat-c2}, we have $\mmax{S}=\ret(P)=\mmax{T'}$, and
$\mmax{R}=\des(P)=\mmax{R'}$.
\end{proof}

\begin{proposition}\label{pro-cat2}
For each $n\ge 0$, the two statistics $\mmax{S}$ and $\mmax{R}$ have symmetric
joint distribution on C1-pairs $(S,R)\in\Co_n$.
\end{proposition}
\begin{proof} This follows from Observation~\ref{obs-stat-c1} and from the
first part of Fact~\ref{fac-catstat}.
\end{proof}

\section{Fishburn triples}\label{sec-fish}

In this section, we will show that objects from certain Fishburn-enumerated
families can be represented by a triple $(T,S,R)$ of relations, satisfying
axioms which generalize the axioms of Catalan pairs. This will allow us to
extend the statistics $\asc$, $\des$, $\ret$ and $\pea$ to Fishburn objects,
where they admit a natural combinatorial interpretation. We will then show, as
the main results of this paper, that some of the classical equidistribution
results listed in Fact~\ref{fac-catstat} can be extended to Fishburn objects.

\begin{definition} A \emph{Fishburn triple} (or \emph{F-triple} for short) is a
relational structure $(T,S,R)$ on a set $X$ satisfying the following axioms:
\item[Fa)] $S$, $R$, and $T\cup R$ are partial orders on $X$.
\item[Fb)] Any two distinct elements of $X$ are comparable by exactly one of
$T$, $S$, or~$R$.
\item[C1c)] For any three distinct elements $x,y,z$ satisfying $xSy$ and $yRz$
we have $xRz$.
\item[C1c*)] For any three distinct elements $x,y,z$ satisfying $xSy$ and $zRy$
we have~$zRx$.
\item[C2c)] There are no three distinct elements $x,y,z\in X$ satisfying $xTy$,
$xTz$ and $yRz$, and also no three elements $x,y,z\in X$ satisfying $yTx$,
$zTx$ and $zRy$.
\end{definition}
As with Catalan pairs, the Fishburn triples may equivalently be described as
structures avoiding certain substructures of size at most three; see
Figure~\ref{fig-forb2}.

\begin{figure}
 \includegraphics[width=\textwidth]{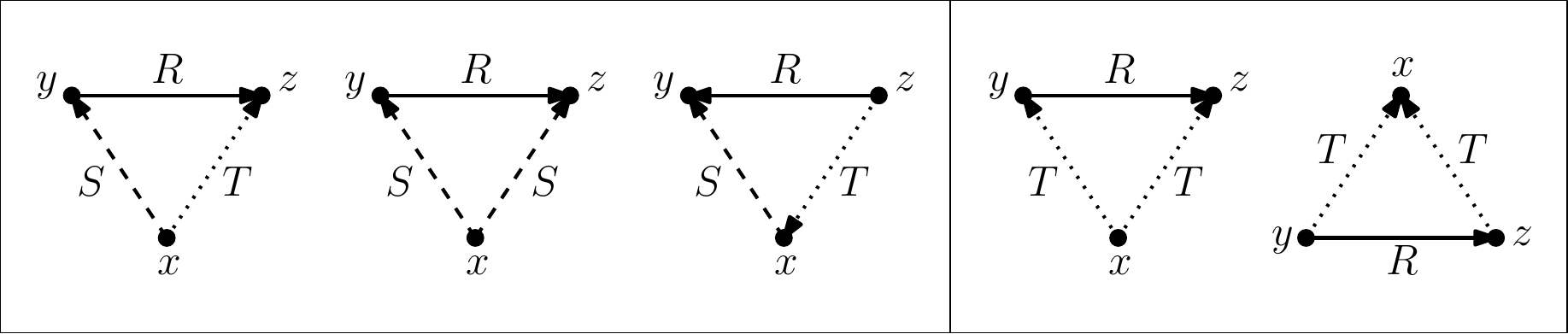}
\caption{Left: the three minimal structures satisfying axioms Fa and Fb, but
not C1c or C1c*. Right: the two structures excluded by axiom
C2c.}\label{fig-forb2}
\end{figure}

Observe that a relational structure $(S,R)$ is a C1-pair if and only if
$(\emptyset,S,R)$ is an F-triple, and a relational structure $(T,R)$ is a
C2-pair if and only if $(T,\emptyset,R)$ is an F-triple. In this way,
F-triples may be seen as a common generalization of the two types of
Catalan pairs. 

Note also that $(T,S,R)$ is an F-triple if and only if $(T^{-1},S, R^{-1})$ is
an F-triple. We will refer to the mapping $(T,S,R)\mapsto(T^{-1},S, R^{-1})$ as
the \emph{trivial involution} on F-triples. We may restrict the trivial
involution to C1-pairs and C2-pairs, with a C1-pair $(S,R)$ being mapped to
$(S,R^{-1})$ and a C2-pair $(T,R)$ being mapped to $(T^{-1},R^{-1})$. When
representing Catalan pairs of either type as Dyck paths, as we did in
Subsection~\ref{ssec-catstats}, the trivial involution acts on Dyck paths of
order $n$ as a mirror reflection whose axis is the line $x+y=n$.

\begin{lemma}\label{lem-fish} If $(T,S,R)$ is an F-triple on a vertex set
$X$, then $R$ is an interval order. Let $M$ be the Fishburn matrix of $R$. Let
$x$ and $y$ two elements of $X$, represented by cells $c_x$ and $c_y$ of~$M$. 
\begin{enumerate}
 \item If $c_x$ is strictly SW of $c_y$, then $(x,y)\in S$.
 \item If $c_x$ is strictly NW of $c_y$, then $(x,y)\in T$.
 \item If $(x,y)\in S$, then $c_x$ is weakly SW of $c_y$.
 \item If $(x,y)\in T$, then $c_x$ is weakly NW of $c_y$.
\end{enumerate}
\end{lemma}
\begin{proof}
Let $(T,S,R)$ be an F-triple on a set~$X$. Let us prove that $R$
is \tptp-free. For contradiction, assume $X$ contains four distinct elements
$x,x',y,y'$, such that $(x,x')$ and $(y,y')$ belong to $R$, while all other
pairs among these four elements are $R$-incomparable. By axiom Fb, $x$ and $y$
are either $S$-comparable or $T$-comparable. However, if $xSy$ holds, then C1c
implies that $xRy'$ holds as well, which is impossible. If $xTy$
holds, then transitivity of $T\cup R$ implies $xTy'$ or $xRy'$. However, $xTy'$
is excluded by axiom C2c and $xRy'$ contradicts the choice of $x,x',y,y'$. This
shows that $R$ is \tptp-free.

We now prove the four numbered claims of the lemma. Let $M$ be the Fishburn
matrix of $R$, and let $x$ and $y$ be two elements of $X$ represented by cells
$c_x=(i_x,j_x)$ and $c_y=(i_y,j_y)$. 

To prove the first claim, suppose that $c_x$ is strictly SW of~$c_y$. By
Lemma~\ref{lem-swse}, there are two elements $u,v\in X$ such that $u,v,x,y$
induce a copy of \tpo, where $(u,x)$, $(x,v)$, and $(u,v)$ belong to $R$, and
$y$ is $R$-incomparable to $x$, $u$, and~$v$. Now $ySx$ would imply $yRv$ by
C1c, which is impossible, since $y$ and $v$ are $R$-incomparable. Next, $yTx$
would imply $yTv$ by transitivity of $T\cup R$, contradicting C2c. Similarly,
$xTy$ implies $uTy$, again contradicting $C2c$. This leaves $xSy$ as the only
option.

To prove the second claim, suppose that $c_x$ is strictly NW of $c_y$. By
Lemma~\ref{lem-swse}, there are elements $u$ and $v$ such that $u$, $v$, $x$ and
$y$ induce a copy of \N\ in $R$, with precisely the three pairs $(x,v)$, $(u,v)$
and $(u,y)$ belonging to~$R$. We want to prove that $(x,y)$ is in $T$. Consider
the alternatives: $yTx$ forces $yTv$ by transitivity of $T\cup R$, contradicting
C2c; on the other hand $xSy$ forces $xRu$ by C1c*, while $ySx$ forces $yRv$ by
C1c, which both contradict the choice of $u$ and $v$. We conclude that $yTx$ is
the only possibility.

For the third claim, proceed by contradiction, and assume that $xSy$ holds, but
$c_x$ is not weakly SW of~$c_y$. This means that $i_x<i_y$ or $j_x>j_y$.
Suppose that $i_x<i_y$, the other case being analogous. Let $c_z$ be a nonzero
cell of $M$ in column $i_x$, and let $z$ be an element represented by~$c_z$. By
the choice of $z$, we know that $zRy$ holds and that $x$ and $z$ are
$R$-incomparable. This contradicts axiom C1c*.

Finally, suppose that $xTy$ holds and $c_x$ is not weakly NW from~$c_y$.
Since $x$ and $y$ are $R$-incomparable, this means that $i_x>i_y$ or $j_x>j_y$.
Suppose that $i_x>i_y$. Let $z$ be an element represented by a cell in column
$i_y$, so that $zRx$ holds, while $y$ and $z$ are $R$-incomparable. Transitivity
of $T\cup R$ implies $zTy$, contradicting C2c.
\end{proof}

Suppose we are given an interval order $R$ on a set $X$, with a corresponding
Fishburn matrix $M$, and we would like to extend $R$ into an F-triple $(T,S,R)$.
The four conditions in Lemma~\ref{lem-fish} put certain constraints on $T$ and
$S$, but in general they do not determine $T$ and $S$ uniquely.
In particular, if $x$ and $y$ are two elements represented by cells $c_x$ and
$c_y$ that belong to the same row or column of $M$, then Lemma~\ref{lem-fish}
does not say whether $x$ and $y$ should be $T$-comparable or $S$-comparable. 

To obtain a unique Fishburn triple for a given $R$, we need to impose additional
restrictions to disambiguate the relations between elements represented by
cells in the same row or column of~$M$. We will consider two ways of imposing
such restrictions. In the first way, all ambiguous pairs end up
$S$-comparable, while in the second way they will be $T$-comparable.

\begin{definition}\label{def-fsft}
Let $R$ be an interval order on a set $X$, and let $M$ be its Fishburn
matrix. The \emph{Fishburn triple of type 1} of $R$ (or \emph{F1-triple of
$R$}) is the relational structure $(T_1,S_1,R)$ on the set $X$, in which $T_1$ 
and $S_1$ are determined by these rules:
\begin{itemize}
 \item For any two elements $x,y\in X$, represented by distinct cells $c_x$ and
$c_y$ in $M$, we have $xT_1y$ if and only if $c_x$ is strictly NW from $c_y$, 
and
we have $xS_1y$ if and only if $c_x$ is weakly SW from $c_y$.
 \item If $c$ is a cell of $M$ and $X_c\subseteq X$ the set of elements
represented by $c$, then $S_1$ induces a chain in $X_c$ and $T_1$ induces an
antichain in~$X_c$.
\end{itemize}
Similarly, for $R$ and $M$ as above, a \emph{Fishburn triple of type 2} of $R$
(or \emph{F2-triple of $R$}) is the relational structure $(T_2,S_2,R)$ on $X$,
with $T_2$ and $S_2$ defined as follows:
\begin{itemize}
 \item For any two elements $x,y\in X$, represented by distinct cells $c_x$ and
$c_y$ in $M$, we have $xT_2y$ if and only if $c_x$ is weakly NW from $c_y$, and
we have $xS_2y$ if and only if $c_x$ is strictly SW from $c_y$.
 \item If $c$ is a cell of $M$ and $X_c\subseteq X$ the set of elements
represented by $c$, then $S_2$ induces an antichain in $X_c$ and $T_2$ induces a
chain in~$X_c$.
\end{itemize}
\end{definition}

Note that the F1-triple and the F2-triple are up to isomorphism uniquely
determined by the interval order~$R$.

\begin{lemma}\label{lem-fsft}
For an interval order $R$, its F1-triple $(T_1,S_1,R)$ and its F2-triple
$(T_2,S_2,R)$ are Fishburn triples. 
\end{lemma}
\begin{proof}
Consider the F1-triple $(T_1,S_1,R)$. Clearly, it satisfies the axioms 
Fa and Fb of F-triples. To check axiom C1c, pick three elements $x,y,z\in 
X$, with $xS_1y$ and $yRz$, and let $c_x=(i_x,j_x)$, $c_y=(i_y,j_y)$ and
$c_z=(i_z,j_z)$ be the corresponding cells of $M$. Then $xS_1y$ implies $j_x\le
j_y$, and $yRz$ implies $j_y<i_z$. Together this proves $j_x< i_z$, and
consequently $xRz$, and axiom C1c holds. Axiom C1c* can be proved by an
analogous argument.

To prove axiom C2c, consider again three elements $x,y,z\in X$ represented by
the cells $c_x=(i_x,j_x)$, $c_y=(i_y,j_y)$ and $c_z=(i_z,j_z)$. To prove
the first part of the axiom, assume for contradiction that $xT_1y$, $xT_1z$ and
$yRz$ holds. Since $c_x$ is strictly NW from $c_y$, we have $j_x<j_y$. From
$yRz$ we get $j_y<i_z$, hence $j_x<i_z$ implying $xRz$, which is a
contradiction. The second part of axiom C2c is proved analogously.

An analogous reasoning applies to $(T_2,S_2,R)$ as well.
\end{proof}

Note that the trivial involution on F-triples maps F1-triples to F1-triples and 
F2-triples to F2-triples.

From Lemma~\ref{lem-swse}, we deduce that an interval order $R$ is
\N-free if and only if its F1-triple has the form $(\emptyset,S,R)$, which
means that $(S,R)$ is a C1-pair, while $R$ is \tpop-free if and only
if its F2-triple has the form $(T,\emptyset,R)$, implying that $(T,R)$ is a
C2-pair. Thus, F1-triples are a generalization of C1-pairs, while F2-triples
generalize C2-pairs. 

Our main goal is to use F1-triples and F2-triples to identify combinatorial
statistics on Fishburn objects that satisfy nontrivial equidistribution
properties. Inspired by the Catalan statistics explored in
Subsection~\ref{ssec-catstats}, we focus on the statistics that can be expressed
as the number of minimal or maximal elements of a component of an F1-triple or
an F2-triple. 

For an interval order $R$, let $(T_1,S_1,R)$ be its F1-triple and 
$(T_2,S_2,R)$ its F2-triple. We may then consider the number of minimal and 
maximal elements in each of the five relations $T_1$, $S_1$, $T_2$, $S_2$, 
and $R$, for a total of ten possible statistics. 

In fact, we have no nontrivial result for the two statistics $\mmin{S_1}$ and
$\mmin{S_2}$. Furthermore, the trivial involution on F-triples maps the minimal
elements of $T_1$ to maximal elements of $T_1$ and vice versa, and the same is 
true for $T_2$ and $R$ as well. Therefore we will not treat the 
statistics $\mmin{T_1}$, $\mmin{T_2}$ and $\mmin{R}$ separately, 
and we focus on the five statistics $\mmax{S_1}$, $\mmax{S_2}$,
$\mmax{T_1}$, $\mmax{T_2}$, and $\mmax{R}$.

To gain an intuition for these five statistics, let us describe
them in terms of Fishburn matrices. Let $(T_1,S_1,R)$ and
$(T_2,S_2,R)$ be as above, and let $M$ be the Fishburn matrix representing the
interval order~$R$. 

Recall from Observation~\ref{obs-fish} that $\mmax{R}$ equals  the weight of the
last column of~$M$. To get a similar description for the remaining four
statistics of interest, we need some terminology. Let us say that a cell $c$ of
the matrix $M$ is \emph{strong-NE extreme cell} (or \emph{sNE-cell} for short),
if $c$ is a nonzero cell of $M$, and any other cell strongly NE from $c$ is a
zero cell. Similarly, $c$ is a \emph{weak-NE extreme cell} (or \emph{wNE-cell})
if it is a nonzero cell and any other cell weakly NE from $c$ is a zero cell.
Note that every weak-NE extreme cell is also a strong-NE extreme cell; in
particular, being strong-NE extreme is actually a weaker property than being
weak-NE extreme. In an obvious analogy, we will also refer to sSE-cells,
wSE-cells, etc.

\begin{observation}\label{obs-extreme}
Let $R$ be an interval order represented by a Fishburn matrix~$M$, let 
$(T_1,S_1,R)$ be its F1-triple and $(T_2,S_2,R)$ be its F2-triple. Then 
\begin{itemize}
 \item $\mmax{S_1}$ is equal to the number of wNE-cells of $M$,
 \item $\mmax{S_2}$ is equal to the total weight of the sNE-cells of $M$,
 \item $\mmax{T_1}$ is equal to the total weight of the sSE-cells of $M$, and
 \item $\mmax{T_2}$ is equal to the number of wSE-cells of $M$.
\end{itemize}
\end{observation}

We are finally ready to state our main results. Let $\F_n$ denote the set of
interval orders on $n$ elements.

\begin{theorem}\label{thm-fish1} Fix $n\ge 0$. There is an involution
$\phi\colon \F_n\to \F_n$ with these properties. Suppose $R\in\F_n$ is an 
interval order with F1-triple $(T_1,S_1,R)$ and F2-triple $(T_2,S_2,R)$. Let 
$R'=\phi(R)$ be its image under $\phi$, with F1-triple $(T'_1,S'_1,R')$ and 
F2-triple $(T'_2,S'_2,R')$. Then the following holds:
\begin{itemize}
 \item $\mmax{S_1}=\mmax{T'_2}$, and hence
$\mmax{T_2}=\mmax{S'_1}$, since $\phi$ is an involution,
 \item $\mmax{S_2}=\mmax{T'_1}$, and hence
$\mmax{T_1}=\mmax{S'_2}$, and
 \item $\mmax{R}=\mmax{R'}$.
\end{itemize}
In other words, the pair of statistics $(\mmax{S_1},\mmax{T_2})$ and the pair of
statistics $(\mmax{S_2},\mmax{T_1})$ both have symmetric joint distribution on
$\F_n$, and the symmetry of both these pairs is witnessed by the same involution
$\phi$ which additionally preserves the value of $\mmax{R}$.
\end{theorem}

\begin{theorem}\label{thm-fish2} Let $(T_1,S_1,R)$ be the F1-triple of an
interval order~$R$. For any $n\ge 0$, the pair of statistics
$(\mmax{S_1},\mmax{R})$ has a symmetric joint distribution over~$\F_n$.
\end{theorem}

In matrix terminology, Theorem~\ref{thm-fish2} states that the statistics
`number of wNE-cells' and `weight of the last column' have symmetric joint
distribution over Fishburn matrices of weight~$n$; see Figure~\ref{fig-biject}.

\begin{figure}
 \hfil\includegraphics{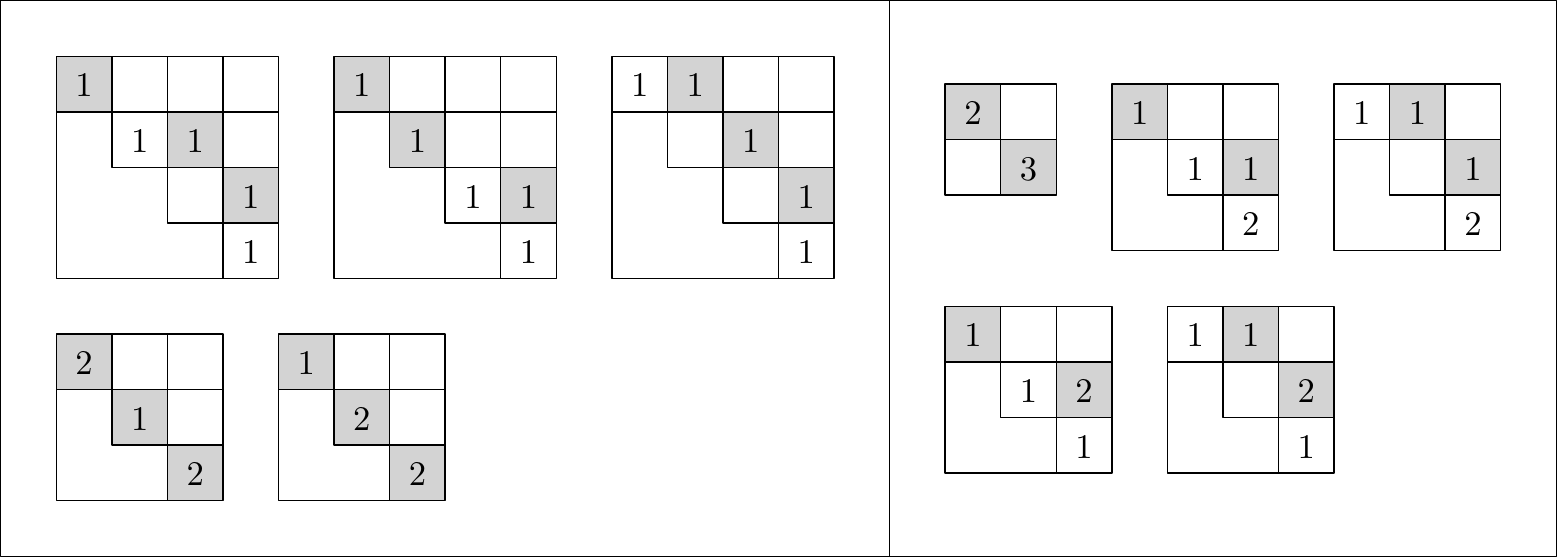}
\caption{Illustration of the symmetric joint distribution of the number of
wNE-cells and the weight of last column over Fishburn matrices. Left: the
Fishburn matrices of weight 5, with 3 wNE-cells, and last column of weight 2.
Right: the Fishburn matrices of weight 5, with 2 wNE-cells, and last column of
weight 3. The wNE-cells are shaded.}\label{fig-biject}
\end{figure}

Note that by combining Theorems~\ref{thm-fish1} and~\ref{thm-fish2}, we may
additionally deduce that $\mmax{R}$ and $\mmax{T_2}$ also have symmetric joint
distribution over~$\F_n$.

Before we present the proofs of the two theorems, let us point out how they 
relate to the results on Catalan statistics discussed previously.
Theorem~\ref{thm-fish1} is a generalization of Proposition~\ref{pro-cat1}. To
see this, consider the situation when $R$ is \tptn-free. With the notation of
Theorem~\ref{thm-fish1}, this implies that $T_1=\emptyset$ and $(S_1,R)$ is a
C1-pair. Consider then the F2-triple $(T'_2,S'_2,R')$. Theorem~\ref{thm-fish1}
states that $\mmax{T_1}=\mmax{S'_2}$, but since $T_1=\emptyset$, it follows
that $S'_2=\emptyset$ as well, since $\emptyset$ is the only relation with $n$
maximal elements. Consequently, $R'$ has an F2-triple of the form
$(T'_2,\emptyset,R')$, hence $(T'_2,R')$ is a C2-pair and $R'$ is \tpto-free. 
We conclude that by restricting the mapping $\phi$ from Theorem~\ref{thm-fish1}
to \tptn-free posets $R$, we get a bijection from C1-pairs to C2-pairs with
the same statistic-preserving properties as in Proposition~\ref{pro-cat1}.

Theorem~\ref{thm-fish2} is inspired by Proposition~\ref{pro-cat2}, and can be
seen as extending the statement of this proposition from C1-pairs to F1-triples.

\subsection{Proofs}

To prove Theorems~\ref{thm-fish1} and \ref{thm-fish2}, it is more convenient to 
work with Fishburn matrices rather than relational structures, and to interpret 
the relevant statistics using Observation~\ref{obs-extreme}.

Recall that an interval order is primitive if it has no two indistinguishable 
elements. Primitive interval orders correspond to Fishburn matrices whose 
entries are equal to 0 or 1; we call such matrices \emph{primitive Fishburn 
matrices}. 

An \emph{inflation} of a primitive Fishburn matrix $M$ is an operation which 
replaces the value of each 1-cell of $M$ (i.e., a cell of value 1) by a
positive integer, while the 0-cells are left unchanged. Clearly, by inflating a
primitive Fishburn matrix we again obtain a Fishburn matrix, and any Fishburn
matrix can be uniquely obtained by inflating a primitive Fishburn matrix.

Another useful operation on primitive Fishburn matrices is the 
\emph{extension}. Informally speaking, it creates a primitive Fishburn 
matrix $P'$ with $k+1$ columns from a primitive Fishburn matrix $P$ with 
$k$-columns, by splitting the last column of $P$ into two new columns.
Formally, suppose that $P=(P_{i,j})_{i,j=1}^k$ is a $k$-by-$k$ primitive 
Fishburn matrix. We say that a $(k+1)$-by-$(k+1)$ matrix 
$P'=(P'_{i,j})_{i,j=1}^{k+1}$ is an extension of $P$, if $P'$ has the following 
properties (see Figure~\ref{fig-extend}):
\begin{itemize}
 \item The last row of $P'$ consists of $k$ 0-cells followed by a 1-cell. In 
other words, for $j\le k$ we have $P'_{k+1,j}=0$, while $P_{k+1,k+1}=1$.
 \item For every $j<k$ and for every $i\le j$, we have $P_{i,j}=P'_{i,j}$. That 
is, the first $k-1$ columns of $P'$ are identical to the first $k-1$ columns 
of $P$, except for an extra 0-cell in the last row.
 \item If $P_{i,k}=0$ for some $i$, then $P'_{i,k}=P'_{i,k+1}=0$. That 
is, each 0-cell in the last column of $P$ gives rise to two 0-cells in the same
row and in the last two columns of~$P'$.
 \item If $P_{i,k}=1$, then there are three options for the values of 
$P'_{i,k}$ and $P'_{i,k+1}$:
\begin{enumerate}
 \item $P'_{i,k}=P'_{i,k+1}=1$. In such case we say that the 1-cell $P_{i,k}$ 
is \emph{duplicated} into $P'_{i,k}$ and $P'_{i,k+1}$.
 \item $P'_{i,k}=0$ and $P'_{i,k+1}=1$. We then say that $P_{i,k}$ is 
\emph{shifted} into $P'_{i,k+1}$.
 \item $P'_{i,k}=1$ and $P'_{i,k+1}=0$. We then say that $P_{i,k}$ is 
\emph{ignored} by the extension.
\end{enumerate}
\end{itemize}

\begin{figure}
\hfil\includegraphics[width=0.5\textwidth]{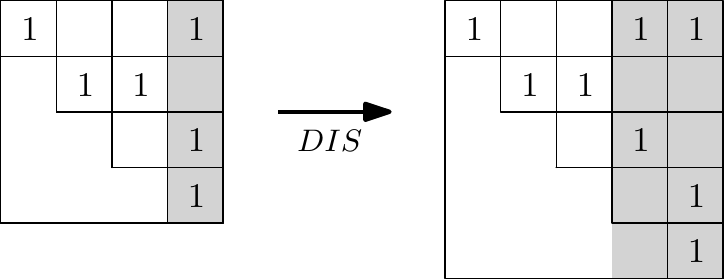}
\caption{Example of an extension of a primitive Fishburn
matrix. The word DIS is the code of the extension. }\label{fig-extend}
\end{figure}

We say that an extension of $P$ into $P'$ is \emph{valid}, if there is at least 
one 1-cell in the penultimate column of $P'$, or equivalently, if at least one 
1-cell in the last column of $P$ has been duplicated or ignored. It is easy to 
see that if $P'$ is a valid extension of a primitive Fishburn matrix $P$, then 
$P'$ is itself a primitive Fishburn matrix, and conversely, any primitive 
Fishburn matrix $P'$ with at least two columns is a valid extension of a unique 
primitive Fishburn matrix~$P$.

Note that a primitive Fishburn matrix $P$ whose last column has weight $m$ has 
exactly $3^m$ extensions; one of them is invalid and $3^m-1$ are valid. For 
convenience, we will represent each extension $P'$ of $P$ by a word 
$w=w_1\dotsb w_m$ of length $m$ over the alphabet $\{D,S,I\}$, defined as 
follows: suppose that $c_1,c_2,\dotsc,c_m$ are the 1-cells in the last column of 
$P$, listed in top-to-bottom order. Then $w_i$ is equal to $D$ (or $S$\, or 
$I$), if the cell $c_i$ is duplicated (or shifted, or ignored, respectively) in 
the extension~$P'$. We will call $w$ the \emph{code} of the extension from 
$P$ to~$P'$. Notice that the 1-cell in the bottom-right corner of $P'$ is not 
represented by any symbol of~$w$.

Given a word $w=w_1w_2\dotsb w_m$ of length $m$, the \emph{reverse} of $w$, 
denoted by $\overline{w}$, is the word $w_m w_{m-1}\dotsb w_1$. 

\begin{observation}\label{obs-extend}
Let $P$ be a $k$-by-$k$ primitive Fishburn matrix with $m$ 1-cells in the last
column, and let $P'$ be its valid extension, with code $w=w_1\dotsb w_m$. Let
$c$ be a 1-cell in the $j$-th column of~$P'$.
\begin{itemize}
 \item Suppose that $j<k$, which implies, in particular, that $c$ is also a
1-cell in~$P$. Then $c$ is an sNE-cell of $P'$ (or wNE-cell of $P'$, or
sSE-cell of $P'$, or wNE-cell of $P'$) if and only if it is a sNE-cell of $P$
(or wNE-cell of $P$, or sSE-cell of $P$, or wNE-cell of $P$, respectively).
 \item Suppose that $j=k$, which means that $c$ is also a 1-cell in $P$, and
this 1-cell was duplicated or ignored by the extension from $P$ to~$P'$. 
Suppose $c$ is the $i$-th 1-cell in the last column of $P$, counted from the
top (i.e., there are $i-1$ 1-cells above $c$ and $m-i$ 1-cells below $c$ in
the last column of $P$). Then $c$ is a wNE-cell of $P'$ if and only if $i=1$ and
$w_1=I$, while $c$ is a wSE-cell of $P'$ if and only if $i=m$ and $w_m=I$.
Furthermore, $c$ is an sNE-cell of $P'$ if and only if all the 1-cells of $P$
above it were ignored (i.e., $w_1=w_2=\dotsb=w_{i-1}=I$), while $c$ is an
sSE-cell of $P'$ if and only if all the 1-cells of $P$ below it were ignored
(i.e., $w_{i+1}=w_{i+2}=\dotsb=w_m=I$).
 \item If $j=k+1$, i.e., $c$ is in the last column of $P'$, then $c$ is an
sNE-cell and also an sSE-cell. Moreover, $c$ is a wNE-cell if and only if 
it is the topmost 1-cell of the last column of~$P'$, and it is a wSE-cell
if and only if it is the bottommost 1-cell of the last column of~$P'$.
\end{itemize}
\end{observation}

\subsubsection{Proof of Theorem~\ref{thm-fish1}}
To prove Theorem~\ref{thm-fish1}, we first describe the involution $\phi$, and 
then verify that it has the required properties. 

Let $M$ be a $k$-by-$k$ Fishburn matrix. As explained above, $M$ can be 
constructed in a unique way from the 1-by-1 matrix \fbox{1} by a sequence of
$k-1$ 
valid extensions followed by an inflation. In particular, there is a sequence
of matrices $P_1, P_2,\dotsc,P_k,M$, where $P_1=\fbox{1}$, for $i>1$ the matrix
$P_i$ is an extension of $P_{i-1}$, and $M$ is an inflation of~$P_k$. 

\newcommand{\oP}{\overline{P}}

Define a new sequence $\oP_1,\dotsc,\oP_k$ as follows:
\begin{itemize}
 \item $\oP_1=P_1=\fbox{1}$.
 \item For each $i>1$, $\oP_i$ is an extension of $\oP_{i-1}$, and 
the code of $\oP_i$ is the reverse of the code of~$P_i$.
\end{itemize}

Observe that for any $1\le j\le i\le k$, the $j$-th column of $\oP_i$ has the 
same weight as the $j$-th column of~$P_i$. Furthermore, from
Observation~\ref{obs-extend}, we immediately deduce that for every $i$ and $j$
such that $1\le j\le i\le k$, the following relationships hold:
\begin{itemize}
 \item The number of sNE-cells in the $j$-th column of $P_i$ equals the number
of sSE-cells in the $j$-th column of~$\oP_i$. 
 \item The number of sSE-cells in the $j$-th column of $P_i$ equals the number
of sNE-cells in the $j$-th column of~$\oP_i$. 
 \item The number of wNE-cells in the $j$-th column of $P_i$ equals the number
of wSE-cells in the $j$-th column of~$\oP_i$. 
 \item The number of wSE-cells in the $j$-th column of $P_i$ equals the number
of wNE-cells in the $j$-th column of~$\oP_i$. 
\end{itemize}

\newcommand{\oM}{\overline{M}}
As the last step in the definition of $\phi$, we describe how to inflate
$\oP_k$ into a matrix $\oM$. Fix a column index $j\le k$. Let $m$ be the number
of 1-cells in the $j$-th column of $P_k$. Since $M$ is an inflation of $P_k$,
it has $m$ nonzero cells in its $j$-th column; let $x_1,x_2,\dotsc,x_m$ be the
weights of these non-zero cells, ordered from top to bottom. As we know,
$\oP_k$ also has $m$ 1-cells in its $j$-th column. We inflate these cells by
using values $x_m,x_{m-1},\dotsc,x_1$, ordered from top to bottom. Doing this
for each $j$, we obtain an inflation $\oM$ of $\oP_k$. We then define $\phi$ by
$\phi(M)=\oM$.

Let us check that $\phi$ has all the required properties. Clearly, $\phi$ is an
involution, and it preserves the weight of the last column (indeed, of any
column). Moreover, the number of wNE-cells of $M$ is equal to the number of
wSE-cells of $\oM$, since these numbers are not affected by inflations. It
remains to see that the total weight of the sNE-cells of $M$ equals the total
weight of the sSE-cells of~$\oM$. Fix a column $j\le k$, and suppose that $M$
has exactly $\ell$ sNE-cells in its $j$-th column. These must be the $\ell$
topmost nonzero cells of the $j$-th column of $M$, and in the notation of the
previous paragraph, their total weight is $x_1+x_2+\dotsb+x_\ell$. It follows
that $P_k$ also has $\ell$ sNE-cells in its $j$-th column, therefore $\oP_k$ has
$\ell$ sSE-cells in its $j$-th column, and these are the bottommost $\ell$
1-cells of the $j$-th column of $\oP_k$. In $\oM$, these $\ell$ cells have total
weight $x_1+\dotsb+x_\ell$. We see that the sNE-cells of $M$ have the same
weight as the sSE-cells of $\oM$, and Theorem~\ref{thm-fish1} is proved.

We remark that the involution $\phi$ actually witnesses more equidistribution
results than those stated in Theorem~\ref{thm-fish1}. E.g., the total weight of
wNE-cells of $M$ equals the total weight of wSE-cells of $\phi(M)$, and the
number of sNE-cells of $M$ equal the number of sSE-cells of~$\phi(M)$. These
facts do not seem to be easy to express in terms of F1-triples or F2-triples.

Moreover, from the construction of $\phi$ it is clear that the conclusions of
Theorem~\ref{thm-fish1} remain valid even when restricted to primitive Fishburn
matrices, or to Fishburn matrices of a given number of rows. No such restriction
is possible in Theorem~\ref{thm-fish2}, as seen from the examples in
Figure~\ref{fig-biject}.

\subsubsection{Proof of Theorem~\ref{thm-fish2}}
As with Theorem~\ref{thm-fish1}, our proof will be based on the concepts of
extension and inflation. However, in this case we are not able to give an
explicit bijection. Instead, we will proceed by deriving a formula for the
corresponding refined generating function.

\newcommand{\w}{w}
\newcommand{\lc}{lc}
\newcommand{\pc}{pc}
\newcommand{\nne}{ne}
\newcommand{\poch}[3]{\left(#1;#2\right)_{#3}}
\newcommand{\pochq}[2]{\left(#1\right)_{#2}}
\newcommand{\pochqn}[1]{\pochq{#1}{n}}

For a matrix $M$, let $\w(M)$ denote its weight, $\lc(M)$ the weight of its last
column, $\pc(M)$ the total weight of the columns preceding the last one (so
$\w(M)=\lc(M)+\pc(M)$), and $\nne(M)$ its number of wNE-cells. Our
goal is to show that the statistics $\lc$ and $\nne$ have symmetric joint
distribution over the set of Fishburn matrices of a given weight.

We will use the standard notation $\poch{a}{q}{n}$ for the
product $(1-a)(1-aq)(1-aq^2)\dotsb(1-aq^{n-1})$.

Let $\M$ be the set of nonempty Fishburn matrices. Our main object of interest
will be the generating function
\[
F(x,y,z)=\sum_{M\in \M} x^{\w(M)} y^{\lc(M)} z^{\nne(M)}.
\]
Theorem~\ref{thm-fish2} is equivalent to the identity $F(x,y,z)=F(x,z,y)$. 
Therefore, the theorem follows immediately from the following proposition.

\begin{proposition}\label{pro-fish2}
 The generating function $F(x,y,z)$ satisfies
\[
 F(x,y,z)=xyz\sum_{n \ge 0} \poch{\vphantom{\big(}(1-xy)(1-xz)}{1-x}{n}.
\]
\end{proposition}

Let us remark that the formula above is a refinement of a previously known
formula for the generating function
\[
G(x,y)=F(x,y,1)=\sum_{M\in\M} x^{\w(M)} y^{\lc(M)},
\]
which also corresponds to refined enumeration of interval orders with respect
to their number of maximal elements (see~\cite[sequence A175579]{oeis}). 
Formulas for $G(x,y)$ have been obtained in different contexts by several 
authors (namely Zagier~\cite{Zagier}, Kitaev and
Remmel~\cite{KitaevRemmel}, Yan~\cite{Yan}, Levande~\cite{Levande},  Andrews
and Jelínek~\cite{AJ}), and there are now three known expressions for this
generating function:
\begin{align*}
 G(x,y)&=\sum_{n\ge 0} \frac{xy}{1-xy}\poch{1-x}{1-x}{n}&
\text{\cite{KitaevRemmel}}\\
&=\sum_{n\ge 1} \poch{1-xy}{1-x}{n} &\text{\cite{Levande, Yan, Zagier}}\\
&=-1\!+\!\sum_{n\ge 0}
pq^n\poch{p}{q}{n}\poch{q}{q}{n} \text{with }p=\frac{1}{1-xy},\, 
q=\frac{1}{1-x}.&\text{\cite{AJ}}
\end{align*}
The second of these three expressions can be deduced from
Proposition~\ref{pro-fish2} by the substitution $z=1$ and a simple manipulation
of the summands. Other authors have derived formulas for generating functions of
Fishburn matrices (or equivalently, interval orders) refined with respect to
various statistics~\cite{BMCDK, indistin, SelfDual}, but to our knowledge, none
of them has considered the statistic~$\nne(M)$.

\begin{proof}[Proof of Proposition~\ref{pro-fish2}]
Our proof is an adaptation of the approach from our previous
paper~\cite{SelfDual}. We first focus on primitive Fishburn matrices. Let
$\cP_k$ be the set of primitive $k$-by-$k$ Fishburn matrices, and let
$\cP=\bigcup_{k\ge 1} \cP_k$. Define auxiliary generating functions
\begin{align*}
 P_k(x,y,z)&=\sum_{M\in\cP_k} x^{\pc(M)} y^{\lc(M)-1} z^{\nne(M)}, \text{ and}\\
 P(x,y,z)&=\sum_{k\ge 1} P_k(x,y,z).
\end{align*}
Since $\lc(M)$ is positive for every $M\in\cP$, the exponent of the factor
$y^{\lc(M)-1}$ in $P_k$ is nonnegative. The idea behind subtracting $1$ from
the exponent is that $y$ now only counts the 1-cells in the last column that
are not wNE-cells, while the unique wNE-cell of the last column is counted
by the variable $z$ only. Note that the wNE-cells of the previous columns
contribute to $x$ as well as to~$z$.

Consider a matrix $M\in \cP_k$, with $\w(M)=n$, $\lc(M)=m$ and $\nne(M)=\ell$.
This matrix contributes with the summand $x^{n-m} y^{m-1} z^\ell$
into~$P_k(x,y,z)$; let us call $x^{n-m} y^{m-1} z^\ell$ the \emph{value} of~$M$.
Let us determine the total value of the matrices obtained by extending~$M$. 
Let $c$ be the topmost 1-cell in the last column of~$M$. If $M'$ is an extension
of $M$, then $\nne(M')=\nne(M)+1$ if $c$ was ignored by the extension, and
$\nne(M')=\nne(M)$ otherwise. Any other 1-cell in the last column may be
ignored, duplicated or shifted, contributing respectively a factor $x$, or $xy$,
or $y$ into the value of $M'$. It follows that the total value of all the
matrices obtained by a valid extension of $M$ is equal to
\begin{multline*}
 xz (x^{n-m}(x+y+xy)^{m-1} z^{\ell}) + xy (x^{n-m}(x+y+xy)^{m-1} z^{\ell})\\ + y
(x^{n-m}(x+y+xy)^{m-1} z^{\ell}) - y x^{n-m}y^{m-1} z^{\ell},
\end{multline*}
where the first three summands are respectively the values of the matrices in
which the cell $c$ has been ignored, duplicated and shifted, and the final
subtracted term is the value of the matrix obtained by the invalid extension.
Note that an extension $M'$ also includes a 1-cell in the bottom-right corner,
which contributes a factor of $y$ into the value of $M'$, but the effect of
this extra factor is cancelled by the fact that the topmost 1-cell in the last
column of $M'$ does not contribute any factor of $y$ into the value of~$M'$.

Summing over all $M\in\cP_k$, and recalling that every matrix in $\cP_{k+1}$
may be uniquely obtained as a valid extension of a matrix in $\cP_k$, we deduce
that
\[
 P_{k+1}(x,y,z)=(xz+xy+y)P_k(x,x+y+xy,z)-yP_k(x,y,z).
\]
Summing this identity over all $k\ge 1$ and noting that $P_1(x,y,z)=z$, we
see that
\begin{align*}
 P(x,y,z)-z &= (xz+xy+y)P(x,x+y+xy,z) - yP(x,y,z), \text{ or equivalently,}\\
P(x,y,z) &= \frac{z}{1+y} + \frac{xz+xy+y}{1+y} P(x,x+y+xy,z). 
\end{align*}
From this functional equation, by a simple calculation 
(analogous to~\cite[proof of Theorem 2.1]{SelfDual}),
we obtain the formula
\begin{align}
 P(x,y,z)&=\sum_{n\ge 0} \frac{z}{(1+x)^n(1+y)}\prod_{i=0}^{n-1} \frac{ 
(1+x)^{i+1}(1+y)-1-x+xz}{(1+x)^i(1+y)}\notag\\
&=\frac{z}{1+y}\sum_{n\ge 0}\poch{\frac{1+x-xz}{(1+x)(1+y)}}{\frac{1}{1+x}}{n}.
\label{eq-p}
\end{align}
Any Fishburn matrix may be uniquely obtained by inflating a primitive
Fishburn matrix, which corresponds to the identity
\[
 F(x,y,z)=\frac{xy}{1-xy}P\left(\frac{x}{1-x},\frac{xy}{1-xy},z\right),
\]
where the factor $\frac{xy}{1-xy}$ on the right-hand side corresponds to the
contribution of the topmost 1-cell in the last column. Substituting
into~\eqref{eq-p} gives
\begin{align*}
 F(x,y,z)&= \frac{xy}{1-xy}\cdot\frac{z}{1+\frac{xy}{1-xy}}\sum_{n\ge 0} 
\poch{\frac{1+\frac{x}{1-x}-z\frac{x}{1-x}}{\left(1+\frac{x}{1-x}
\right)\left(1+\frac{xy}{1-xy}\right)}}{\frac{1}{1+\frac{x}{1-x}}}{\!n}\\
&=xyz\sum_{n\ge 0}\poch{\vphantom{\big(}(1-xy)(1-xz)}{1-x}{n}.
\end{align*}
This completes the proof of Proposition~\ref{pro-fish2}, and of
Theorem~\ref{thm-fish2}.
\end{proof}

\section{Further Remarks and Open Problems}

\paragraph{Diagonal-free fillings of polyominoes.} We have seen in
Lemma~\ref{lem-swse} that Fishburn matrices with no two positive cells in
strictly SW position correspond to \tpto-free posets, while those with no two
positive cells in strictly SE position correspond to \tptn-free posets. Both
classes are known to be enumerated by Catalan numbers, and in particular,
the two types of matrices are equinumerous. This can be seen as a very special
case of symmetry between fillings of polyomino shapes avoiding increasing or
decreasing chains. 

To be more specific, a $k$-by-$k$ Fishburn matrix (or indeed, any
upper-triangular $k$-by-$k$ matrix) can also be represented as a
right-justified array of boxes with $k$ rows of lengths $k, k-1, k-2,\dotsc, 1$,
where each box is filled by a nonnegative integer. Such arrays of integers are a
special case of the so-called \emph{fillings of polyominoes}, which are a rich 
area of study. 

Given a Fishburn matrix $M$, we say that a $m$-tuple of nonzero cells
$c_1,\dotsc,c_m$ forms an \emph{increasing chain} (or \emph{decreasing chain})
if for each $i<j$, the cell $c_j$ is strictly NE from $c_i$ (strictly SE from
$c_i$, respectively). It follows from general results on polyomino fillings
(e.g. from a theorem of Rubey~\cite[Theorem 5.3]{Rubey}) that there is a
bijection which maps Fishburn matrices avoiding an increasing chain of length
$m$ to those that avoid a decreasing chain of length $m$, and the bijection
preserves the number of rows, as well as the weight of every row and column. 

We have seen that certain equidistribution results related to Fishburn matrices
without increasing chains of length 2 can be extended to general Fishburn
matrices. It might be worthwhile to look at Fishburn matrices avoiding chains of
a given fixed length $m>2$, and see if these matrices exhibit similar kinds of
equidistribution.

\paragraph{F-triples for other objects.} Disanto et al.~\cite{DisantoAAM}
have shown that many Catalan-enumerated objects, such as Dyck paths,
312-avoiding permutations, non-crossing partitions or non-crossing matchings,
admit a natural bijective encoding into C1-pairs. For other Catalan objects,
e.g. non-nesting matchings, an encoding into C2-pairs is easier to obtain.

Since F1-triples and F2-triples are a direct generalization of C1-pairs and
C2-pairs, it is natural to ask which Fishburn-enumerated objects admit an easy
encoding into such triples. We have seen that such an encoding exists for
interval orders, as well as for Fishburn matrices. It is also easy to
describe an F1-triple structure on non-neighbor-nesting
matchings~\cite{Stoimenow,CL}, since those are closely related to Fishburn
matrices.

A more challenging problem is to find an F-triple structure on ascent
sequences. Unlike other Fishburn objects, ascent sequences do not exhibit any
obvious `trivial involution' that would be analogous to duality of
interval orders or transposition of Fishburn matrices. 

\begin{problem}
Find a natural way to encode an ascent sequence into an F-triple. 
\end{problem}
There are several results showing that certain subsets of ascent sequences are
enumerated by Catalan numbers (see e.g. Duncan and Steingrímsson~\cite{DS},
Callan et al.~\cite{CMS}, or Mansour and Shattuck~\cite{Mansour}). It might be a
good idea to first try to encode those Catalan-enumerated classes into C1-pairs
or C2-pairs.

Another Fishburn-enumerated family for which we cannot find an F-triple encoding
is the family of $\ppat$-avoiding permutations. We say that a permutation
$\pi=\pi_1\dotsb\pi_n$ is $\ppat$-avoiding, if there are no three indices
$i,j,k$ such that $i+1=j<k$ and $\pi_i+1=\pi_k<\pi_j$.
Parviainen~\cite{bivincular} has shown that $\ppat$-avoiding permutations are
enumerated by Fishburn numbers. 

Numerical evidence suggests that there is a close connection between the
statistics of Fishburn matrices expressible in terms of F-triples, and certain
natural statistics on $\ppat$-avoiding permutations. In a permutation
$\pi=\pi_1\pi_2\dotsb\pi_n$, an element $\pi_i$ is a \emph{left-to-right
maximum} (or \emph{LR-maximum}) if $\pi_i$ is larger than any element among
$\pi_1,\dotsc,\pi_{i-1}$. Let $\LRmax(\pi)$ denote the number of LR-maxima
of~$\pi$. Analogously, we define LR-minima, RL-maxima and RL-minima of $\pi$,
their number being denoted by $\LRmin(\pi)$, $\RLmax(\pi)$ and $\RLmin(\pi)$,
respectively. Let $\Av_n(\pat)$ be the set of $\ppat$-avoiding permutations of
size~$n$. Observe that for a $\ppat$-avoiding permutation $\pi$, its composition
inverse $\pi^{-1}$ avoids $\ppat$ as well. Let $\M_n$ be the set of
Fishburn matrices of weight~$n$.

\begin{conjecture}\label{con-pat1} For every $n$, there is a bijection 
$\phi\colon \Av_n(\pat) \to \M_n$ with these properties:
\begin{itemize}
 \item $\LRmax(\pi)$ is the weight of the first row of $\phi(\pi)$,
 \item $\RLmin(\pi)$ is the weight of the last column of $\phi(\pi)$,
 \item $\RLmax(\pi)$ is the number of wNE-cells of $\phi(\pi)$,
 \item $\LRmin(\pi)$ is the number of positive cells of $\phi(\pi)$ belonging to
the main diagonal, and
 \item $\phi(\pi^{-1})$ is obtained from $\phi(\pi)$ by transpose along the
North-East diagonal. 
\end{itemize}
\end{conjecture}

Together with Theorem~\ref{thm-fish2}, Conjecture~\ref{con-pat1} implies the
following weaker conjecture.

\begin{conjecture}\label{con-pat2}
 For any $n$, $\LRmax$ and $\RLmax$ have symmetric joint
distribution on~$\Av_n(\pat)$.
\end{conjecture}

\paragraph{Primitive Fishburn matrices and Motzkin numbers.} Recall that an
interval order is primitive if it has no two indistinguishable elements.
Primitive interval orders are encoded by primitive Fishburn matrices which are
Fishburn matrices whose every entry is equal to 0 or 1. Primitive interval
orders (and therefore also primitive Fishburn matrices) were enumerated by Dukes
et al.~\cite{indistin}, see also~\cite[sequence A138265]{oeis}. As we pointed
out before, every Fishburn matrix can be uniquely obtained from a primitive
Fishburn matrix by inflation. Thus, knowing the enumeration of primitive objects
we may obtain the enumeration of general Fishburn objects and vice versa.

In particular, with the help of Lemma~\ref{lem-swse}, we easily deduce that the
number of primitive \tpto-free interval orders of size $n$, as well as the
number of primitive \tptn-free interval orders of size $n$ is the $(n-1)$-th
Motzkin number (see~\cite[A001006]{oeis}). Maybe this connection between Motzkin
objects and primitive Fishburn objects will reveal further combinatorial
results, in the same way the connection between Catalan objects and Fishburn
objects inspired Theorems~\ref{thm-fish1} and~\ref{thm-fish2}.

\paragraph{Tamari-like lattices.} Disanto et al.~\cite{Disanto_TamariArxiv}
defined a lattice structure on the set of interval orders of a given size, and
showed that the restriction of this lattice to the set of \tptn-free posets
yields the well-known Tamari lattice of Catalan objects~\cite{FriedmanTamari,
HuangTamari}. They also gave a simple description of the Tamari lattice on
\tptn-free posets~\cite{tamari}.

It might be worth exploring whether the lattice structure introduced by
Disanto et al. for interval orders admits an easy interpretation in terms of
Fishburn triples or Fishburn matrices. It might also be worthwhile to look
at the restriction of this lattice to \tpto-free posets.

\paragraph{The missing bijection.} For Theorem~\ref{thm-fish2}, we could provide
a proof by a generating function argument, but not a direct bijective proof.
By finding a bijective proof, or even better, a proof that generalizes
Deutsch's~\cite{DeutschInvol} involution on Catalan objects, we might, e.g.,
gain more insight into the statistic $\mmin{S_1}$, which generalizes the
Narayana-distributed statistic $\pea$ on Dyck paths. 
\begin{problem}
 Prove Theorem~\ref{thm-fish2} bijectively.
\end{problem}

{
\setlength{\bibsep}{0.0pt}
\small
\bibliographystyle{plain}
\bibliography{intord}
}
\end{document}